\def\draft{n}
\documentclass[12pt]{amsart}
\pdfoutput=1
\usepackage[headings]{fullpage}
\usepackage{amssymb,epic,eepic,epsfig}
\usepackage{graphicx}
\usepackage{texdraw}
\usepackage{url}
\usepackage{subfigure}
\usepackage{multirow}
\usepackage{array}
\usepackage[bookmarks=true,%
    colorlinks=true,%
    linkcolor=blue,%
    citecolor=blue,%
    filecolor=blue,%
    menucolor=blue,%
    urlcolor=blue,%
    breaklinks=true]{hyperref}

\usepackage{enumerate}
  
\usepackage{mathdots}
\usepackage{amsfonts}

\usepackage{caption}

\usepackage[all]{xy}

\newtheorem{theorem}{Theorem}[section]
\theoremstyle{definition}
\newtheorem{proposition}[theorem]{Proposition}
\newtheorem{lemma}[theorem]{Lemma}
\newtheorem{definition}[theorem]{Definition}
\newtheorem{remark}[theorem]{Remark}
\newtheorem{corollary}[theorem]{Corollary}

\newtheorem{example}[theorem]{Example}

\newtheorem{notation}[theorem]{Notation}
\newtheorem{observation}[theorem]{Observation}

\numberwithin{equation}{section}

\def\printname#1{
        \if\draft y
                \smash{\makebox[0pt]{\hspace{-0.5in}
                        \raisebox{8pt}{\tt\tiny #1}}}
        \fi
}

\newlength{\standardunitlength}
\catcode`\@=11
\long\def\@makecaption#1#2{%
     \vskip 10pt

\setbox\@tempboxa\hbox{
       \small\sf{\bfcaptionfont #1. }\ignorespaces #2}%
     \ifdim \wd\@tempboxa >\captionwidth {%
         \rightskip=\@captionmargin\leftskip=\@captionmargin
         \unhbox\@tempboxa\par}%
       \else
         \hbox to\hsize{\hfil\box\@tempboxa\hfil}%
     \fi}
\font\bfcaptionfont=cmssbx10 scaled \magstephalf
\newdimen\@captionmargin\@captionmargin=2\parindent
\newdimen\captionwidth\captionwidth=\hsize
\catcode`\@=12

\def\truncsimp{\overline{\Delta}}
\def\dtruncsimp{\overline{\overline{\Delta}}}
\def\cxymatrix#1{\xy*[c]\xybox{\xymatrix#1}\endxy}

\def\BZ{\mathbb Z}

\def\BQ{\mathbb Q}

\def\BC{\mathbb C}

\def\a{\alpha}

\def\e{\epsilon}

\def\b{\beta}

\def\longto{\longrightarrow}

\def\SL{\mathrm{SL}}
\def\PSL{\mathrm{PSL}}
\def\PGL{\mathrm{PGL}}
\def\GL{\mathrm{GL}}

\def\sgn{\mathrm{sgn}}

\DeclareMathOperator{\Conj}{Conj}
\DeclareMathOperator{\id}{id}
\DeclareMathOperator{\sub}{sub}

\def\be{  \begin{equation} }
\def\ee{  \end{equation} }

\def\diag{\text{diag}}

\def\CT{\mathcal T}

\def\R{\mathbb{R}}
\def\Z{\mathbb{Z}}
\def\C{\mathbb{C}}

\def\Q{\mathbb{Q}}
\def\T{\mathcal{T}}

\usepackage{ifpdf}
\ifpdf
\fi

\begin{document}

\begin{abstract}
In \cite{GaroufalidisThurstonZickert} we parametrized boundary-unipotent representations of a 3-manifold group into $\SL(n,\C)$ using \emph{Ptolemy coordinates}, which were inspired by $\mathcal A$-coordinates on higher Teichm\"{u}ller space due to Fock and Goncharov. In this paper, we parametrize representations into $\mathrm{PGL}(n,\C)$ using \emph{shape coordinates} which are a 3-dimensional analogue of Fock and Goncharov's $\mathcal X$-coordinates. These coordinates satisfy equations generalizing Thurston's gluing equations. These equations are of Neumann-Zagier type and satisfy symplectic relations with applications in quantum topology. We also explore a duality between the Ptolemy coordinates and the shape coordinates.
\end{abstract}


\title[Gluing equations for $\PGL(n,\BC)$-representations of 3-manifolds]{
Gluing equations for $\PGL(n,\BC)$-representations of 3-manifolds}
\author{Stavros Garoufalidis}
\address{School of Mathematics \\
         Georgia Institute of Technology \\
         Atlanta, GA 30332-0160, USA \newline
         {\tt \url{http://www.math.gatech.edu/~stavros }}}
\email{stavros@math.gatech.edu}
\author{Matthias Goerner}
\address{University of Maryland \\
         Department of Mathematics \\
         College Park, MD 20742-4015, USA \newline
         {\tt \url{http://math.berkeley.edu/~matthias/}}}
\email{enischte@gmail.com}
\author{Christian K. Zickert}
\address{University of Maryland \\
         Department of Mathematics \\
         College Park, MD 20742-4015, USA \newline
         {\tt \url{http://www2.math.umd.edu/~zickert}}}
\email{zickert@umd.edu}
\thanks{S.~G.~and C.~Z.~were supported in part by the NSF. \\
\newline
1991 {\em Mathematics Classification.} Primary 57N10. Secondary 57M27.
}

\date{July 25, 2012}

\maketitle

\tableofcontents

\section{Introduction}

\subsection{Thurston's gluing equations}
\label{sub.ge}

Thurston's gluing equations are a system of polynomial equations defined for 
a compact $3$-manifold $M$ together with a topological ideal triangulation 
$\T$ of the interior of $M$. The gluing equations were introduced to concretely construct a complete hyperbolic structure on $M$ from a (suitable) solution to the gluing equations. 

Although Thurston only considered manifolds whose boundary components are tori (a necessary condition for the existence of a solution yielding a hyperbolic structure), the gluing equations are defined for manifolds with arbitrary (possibly empty) boundary. The existence of hyperbolic structures is of no concern to us here.

The gluing equations consist of an \emph{edge equation} for each $1$ cell of $\T$ and a \emph{cusp equation} for each generator of the fundamental group of each boundary component of $M$. The system may be written in the form
\begin{equation}\label{eq:GluingEqIntro}
\prod_j z_j^{A_{ij}} \prod_j (1-z_j)^{B_{ij}}=\e_i,
\end{equation}
where $A$ and $B$ are matrices whose columns are parametrized by the simplices of $\T$ and $\e_i$ is a sign (which is $1$ for the edge equations). Each variable $z_j$ may be thought of as an assignment of an ideal simplex shape to a simplex of $\T$. If the shapes $z_j\in\C\setminus\{0,1\}$ satisfy~\eqref{eq:GluingEqIntro}, as well as some extra conditions on the arguments of $z_i$, the ideal simplices glue together to form a complete hyperbolic structure on $M$. Ignoring the cusp equations gives structures that are incomplete. This gives rise to an efficient algorithm for constructing hyperbolic structures, which has been effectively implemented in software packages such as SnapPea~\cite{SnapPea}, Snap~\cite{snapprogram}, and SnapPy~\cite{SnapPy}. 

Among the numerous important features of the gluing equations we will focus
on two:
\begin{itemize}
\item[(a)] The symplectic property of the exponent matrix $(A \vert B)$
of the gluing equations due to Neumann and Zagier~\cite{NeumannZagier}.
\item[(b)] The link to $\PGL(2,\BC)$ representations via a 
\emph{developing map}
\begin{equation}\label{eq.dev2}
V_2(\CT) \to \{\rho\colon\pi_1(M)\to\PGL(n,\C)\}\big/\Conj
\end{equation}
where  $V_2(\T)$ denotes the affine variety of solutions in 
$\C\setminus\{0,1\}$ to the \emph{edge} equations, and the right hand side denotes the set of conjugacy classes of representations of $\pi_1(M)$ in $\PGL(2,\C)$. 

\end{itemize}
We may thus think of $V_2(\T)$ as a parametrization of representations. Note, however, that $V_2(\T)$ depends on the triangulation, and that the developing map need neither be onto nor finite to one. However, if the triangulation is sufficiently fine (a single barycentric subdivision suffices) the developing map is onto, i.e.~every representation is detected (including reducible ones). A solution satisfying the cusp equations as well gives rise to a representation that is boundary-unipotent, i.e.~takes peripheral curves to unipotent elements. Our goal is to
\begin{itemize}
\item
Extend Thurston's gluing equations to $\PGL(n,\BC)$ preserving the above
features for $n=2$, using suitable shape parameters.
\item
Relate the Ptolemy parameters of~\cite{GaroufalidisThurstonZickert}
to the shape parameters via a monomial map.
\end{itemize}

The Ptolemy and shape coordinates were inspired by the $\mathcal A$-coordinates
and $\mathcal X$-coordinates on higher Teichm\"{u}ller spaces due to Fock and 
Goncharov~\cite{FockGoncharov}. Note, however, that Fock and 
Goncharov study surfaces, whereas we study 3-manifolds. Shape coordinates for $n=3$ have been studied independently by Bergeron, Falbel and Guilloux~\cite{FrenchPeople}.

There is a very interesting interplay between the shape coordinates and the Ptolemy coordinates. The groups $\PGL(n,\C)$ and $\SL(n,\C)$ are Langlands dual, and we believe that this interplay is a 3-dimensional aspect of the duality discussed for surfaces by Fock and Goncharov~\cite[p.~33]{FockGoncharov}. The duality is particularly explicit when all the boundary components of $M$ are tori, see Proposition~\ref{prop:DualityTori}.

While the Ptolemy variety naturally parametrizes boundary-unipotent representations, the gluing equations also parametrize representations that are not necessarily boundary-unipotent. The boundary-unipotent ones can be determined by adding additional equations, which are generalizations of Thurston's cusp equations. This is studied in Section~\ref{sec:CuspEquations}.

\subsection{Our main results}
\label{sub.results}
Given a topological ideal triangulation $\CT$, each simplex of $\T$ is divided into $\binom{n+1}{3}$ overlapping \emph{subsimplices} (see Definition~\ref{def:Subsimplex}) and each edge of each subsimplex is assigned a \emph{shape parameter} (see Definition~\ref{shapeassignment}). These are the variables of the gluing equations. There is an equation for each non-vertex integral point 
point of $\T$ (see Definition \ref{def:IntegralPoint}). These are given in Definition~\ref{def:GluingEquationsOri}.

We can write the gluing equations (without the cusp equations) in the form
\begin{equation}\label{eq:GluinEqSimplified}
\prod_s z_s^{A_{p,s}}\prod_s(1-z_s)^{B_{p,s}}=1,
\end{equation}
where $A$ and $B$ are matrices whose rows are parametrized by the non-vertex integral points $p$ of $\T$ and whose colums are parametrized by the subsimplices $s$ of $\T$.

\begin{theorem}\label{thm:soft} Let $P=(A\vert B)$ be the concatenation of the matrices $A$ and $B$ in~\eqref{eq:GluinEqSimplified}.

\begin{itemize}
\item[(a)]\label{item:Soft1} The rows of $P$ Poisson commute, i.e.~for any two rows $v$ and $w$, $\langle v,w\rangle=0$, where $\langle,\rangle$ is the symplectic form given by $\left(\begin{smallmatrix}0&I\\-I&0\end{smallmatrix}\right)$.
\item[(b)]\label{item:soft2} If all boundary components of $M$ are tori, $P$ is an $r\times 2r$ matrix, where $r=t\binom{n+1}{3}$ and $t$ is the number of simplices of $\T$.
\end{itemize}
\end{theorem}

Let $V_n(\CT)$ denote the affine variety of solutions to the 
$\PGL(n,\BC)$ gluing equations and let $P_n(\CT)$ denote the affine variety
of solutions to the Ptolemy equations of~\cite{GaroufalidisThurstonZickert} (see Section~\ref{sec:PtolemyReview} for a review).
The link to representations is given by the result below, which also gives the relationship between the shape coordinates and the Ptolemy coordinates.

\begin{theorem}\label{thm:ShortMainThm} There is a monomial map $\mu\colon P_n(\T)\to V_n(\T)$ which fits in a commutative diagram
\begin{equation}\label{eqn:ShortDiagramIntro}
\cxymatrix{{{P_n(\T)}\ar[r]^-{\mathcal R}\ar[d]^\mu&{\left\{\txt{$\rho\colon\pi_1(M)\to\SL(n,\C)$\\\textnormal{ boundary-unipotent}}\right\}\Big/\Conj}\ar[d]^\pi\\{V_n(\T)}\ar[r]^-{\mathcal R}&{\{\rho\colon\pi_1(M)\to\PGL(n,\C)\}\big/\Conj}}}
\end{equation}
where the map $\pi$ is induced by the canonical map $\SL(n,\C)\to\PGL(n,\C)$. Furthermore, the horizontal maps are surjective if the triangulation $\T$ is sufficiently fine.
\end{theorem}




Theorem~\ref{thm:ShortMainThm} is an immediate consequence of Theorem~\ref{thm:MainThm} below which displays some more of the underlying structure. Briefly, a \emph{decoration} is an equivariant assignment of a coset to each vertex of each simplex of $\T$ (see Definition~\ref{def:Decoration}) and a \emph{cocycle} is an assignment of matrices to the edges satisfying the standard cocycle condition that the product around each face is $1$ (see Definition~\ref{def:GCocycle}). Generic decorations and natural cocycles are defined in Definition~\ref{def:generic} and Definitions~\ref{def:NaturalSLCocycle} and \ref{def:NaturalPGLCocycle}. 

\begin{theorem}\label{thm:MainThm}
There is a commutative diagram
\begin{equation} \label{eqn:DiagramIntro}
\cxymatrix{{
\left\{\txt{\textnormal{Ptolemy}\\\textnormal{assignments}}\right\} 
\ar[d]^\mu & \ar[d]^\pi \ar_-{\mathcal C}[l] \left\{\txt{\textnormal{Generic} $\SL(n,\C)/N$\\\textnormal{decorations}}\right\} \ar^-{\mathcal L_{\alpha\beta}}[r] & \ar[d]^\tau \left\{\txt{\textnormal{Natural} $(\SL(n,\C),N)$\\\textnormal{cocycles}}\right\}\\
\left\{\txt{\textnormal{Shape}\\\textnormal{assignments}}\right\} 
& \ar_-{\mathcal Z}[l] \left\{\txt{\textnormal{Generic} $\PGL(n,\C)/B$\\\textnormal{decorations}}\right\} \ar^-{\mathcal L_{\alpha\beta\gamma}}[r] & \left\{\txt{\textnormal{Natural} $(\PGL(n,\C),B,H)$\\\textnormal{cocycles}}\right\}
}}
\end{equation}
in which the horizontal maps are $1$-$1$-correspondences. All maps are explicit with explict inverses and respect the symmetries of a simplex.
\end{theorem}
The fact that the top horizontal maps are $1$-$1$-correspondences was proved in Garoufalidis-Thurston-Zickert~\cite{GaroufalidisThurstonZickert}.

To see that Theorem~\ref{thm:MainThm} implies Theorem~\ref{thm:ShortMainThm}, note that a cocycle determines a representation by picking a base point and taking products along edge paths. Furthermore, a decoration also determines a representation using the dual triangulation of $\T$ which is generated by face pairings. The last statement of Theorem~\ref{thm:ShortMainThm} follows from Remark~\ref{rm:FineTriangulation}. 


\subsection{Computations and applications}
\label{sub.computations}

The gluing equations of an ideal triangulation is a standard object of 
SnapPy~\cite{SnapPy}, which is used to study invariants of hyperbolic $3$-manifolds. From the gluing equations, one can compute the so-called
\emph{Neumann-Zagier datum} of an ideal triangulation, i.e.~a triple $((A|B),z,f)$ that consists of the matrices $(A|B)$ of
the gluing equations, a shape solution $z$, and a choice of flattening $f$.
There are three recent applications of the Neumann-Zagier datum in quantum topology:
the quantum Riemann surfaces of \cite{Di1}, the loop invariants of \cite{DG}
and the 3D index of \cite{DGG2} (see also \cite{Ga-3Dindex}). These 
applications are reviewed in Section \ref{sub.applicationsQT}, and lead
to exact computations.

Our generalized gluing equations for $\PGL(n,\BC)$ have
been coded into SnapPy by the second author and will be available
in the next release of SnapPy. As an application, we can define and efficiently
compute the $\PGL(n,\BC)$ Neumann-Zagier datum of an ideal triangulation.
Every function of the $\PGL(2,\BC)$ Neumann-Zagier datum can be evaluated
at the $\PGL(n,\BC)$ Neumann-Zagier datum. Sample computations of the one-loop invariant
of the $\PGL(n,\BC)$ Neumann-Zagier datum (the former being an element
of the invariant trace field) are given in Section 
\ref{sub.applicationsQT}. 

Even for $n=2$, our results provide new data. Preexisting software such as SnapPea~\cite{SnapPea}, Snap~\cite{snapprogram}, and SnapPy~\cite{SnapPy} all solve the gluing equations numerically (exact computations can then be guessed using the LLL algorithm), but only give the shapes for the geometric representation. For the Ptolemy varieties exact computations are possible for $n=2$ even when there are many simplices, and there are often several components of representations besides the geometric one. We should point out that while Gr\"obner basis computations are feasible for the Ptolemy varieties even for many simplices, they are usually impractical for the gluing equations even when the cusp equations are added. However, the monomial map $\mu$ can be used to obtain shapes from Ptolemy coordinates. All our tools will be available in the upcoming release of SnapPy.

\subsection{Overview of the Paper}\label{sub:Overview}

In Section~\ref{sec:TriangulationSection} we define the notion of a \emph{concrete triangulation}, which is a triangulation together with a vertex ordering of each simplex. Two types of concrete triangulations are particularly important, \emph{oriented} triangulations and \emph{ordered} triangulations. 
In Section~\ref{sec:ThurstonReview} we review Thurson's gluing equations, and 
in Section~\ref{sec:GeneralizedGluingEq} we define the analogues for $n\geq 2$. The key notion is that of a \emph{shape assignment}, which is defined first for a simplex and later for a triangulation. A shape assignment on a triangulation is a shape assignment on each simplex, such that the shapes satisfy the generalized gluing equations. 
In Section~\ref{sec:PtolemyReview} we review the theory of Ptolemy coordinates developed in~\cite{GaroufalidisThurstonZickert}, and in Section~\ref{sec:PtolemyToShapes} we define a map $\mu$ from Ptolemy assignments to shape assignments. 
In section~\ref{sec:Symplectic} we prove Theorem~\ref{thm:soft} and discuss some applications in quantum topology.
Sections~\ref{sec:Decorations}-\ref{sec:ReconstructRep} are devoted to proving Theorem~\ref{thm:MainThm}.
In Section~\ref{sec:Decorations}, we briefly review the notion of a decoration, and define the maps $\mathcal C$ and $\mathcal Z$ in \eqref{eqn:DiagramIntro}. In Section~\ref{sec:NaturalCocycles} we define the notion of a natural cocycle, and define the maps $\mathcal L_{\alpha\beta}$ and $\mathcal L_{\alpha\beta\gamma}$. In Section~\ref{sec:ExplicitFormulas}, we show that the natural cocycle of a decoration is given explicitly in terms of the shapes (or Ptolemy coordinates), and in Section~\ref{sec:ReconstructRep} we show that the bottom maps of \eqref{eqn:DiagramIntro} are bijective concluding the proof of Theorem~\ref{thm:MainThm}. In Section~\ref{sec:Duality} we discuss a duality between Ptolemy coordinates and shape coordinates, and in Section~\ref{sec:CuspEquations} we show how to add cusp equations to ensure that the representations are boundary-unipotent. Finally, in Section~\ref{sec:Example} we write down the gluing equations and cusp equations for the figure-eight knot complement for $n=3$ and $n=4$.

\begin{remark}\label{rm:ObstructionCocycles}
The $\SL(2,\C)$-Ptolemy varieties are often empty for the cusped census manifolds. Even though the geometric representation of a cusped hyperbolic manifold lifts to $\SL(2,\C)$, no lift is boundary-unipotent, and often (non-trivial) boundary-unipotent $\SL(2,\C)$-representations don't exist.
In Garoufalidis-Thurston-Zickert, we also considered Ptolemy varieties for $p\SL(n,\C)=\SL(n,\C)\big/\pm I$, defined when $n$ is even via an obstruction class in $H^2(M,\partial M;\Z/2\Z)$. The primary purpose of this was to ensure that the image of the geometric representation under the unique irreducible representation $\PSL(2,\C)\to p\SL(n,\C)$ is detected for all census manifolds (more generally, for triangulations where all edges are essential). In this paper we shall only consider the $\SL(n,\C)$-Ptolemy variety. 
One can develop all the theory using the $p\SL(n,\C)$-Ptolemy varieties, but since our main interest here is in the shape coordinates (and for clarity of exposition), we shall not do this here.
\end{remark}

\begin{remark}\label{rm:Volume}
In Garoufalidis-Thurston-Zickert~\cite{GaroufalidisThurstonZickert} we defined the volume (in fact, complex volume) of a boundary-unipotent $\SL(n,\C)$-representation and gave an explicit formula using the Ptolemy coordinates. Similary, one can define the volume of a decorated $\PGL(n,\C)$-representation by adding the volumes of each of the shapes. The volume is an invariant of a decorated $\PGL(n,\C)$-representation (in the sense of Remark~\ref{rm:DecorationsAndBundles}), but we do not know if the volume is independent of the decoration. This is non-trivial even for $n=2$, where it was first proved by Francaviglia~\cite{FrancavigliaVolume}. We shall not deal with this here. 
\end{remark}

\subsection{Acknowledgement} The authors wish to thank Nathan Dunfield, Walter Neumann and Dylan Thurston for helpful comments.

\section{Concrete triangulations}\label{sec:TriangulationSection}

In all of the following $M$ denotes a compact, oriented $3$-manifold with (possibly empty) boundary. Let $\widehat M$ be the space obtained from $M$ by collapsing each boundary component of $M$ to a point. An \emph{ordered simplex} is a simplex together with an ordering of its vertices.

\begin{definition}
An \emph{abstract triangulation} $\mathcal{T}$ of $M$ is an identification of $\widehat{M}$ with a space obtained from a finite collection of $3$-simplices by gluing together pairs of faces via \emph{face-pairings}, i.e.~affine homeomorphisms. A \emph{concrete triangulation} $\mathcal{T}$ of $M$ is an abstract triangulation together with a fixed identification of each $3$-simplex with a standard ordered $3$-simplex.
\end{definition}

The advantage of a concrete triangulation is that each simplex inherits a vertex ordering from the standard simplex. This extra datum gives us a concrete indexing scheme for the vertices and edges and allows us to concretely write down defining equations for the gluing equation variety and the Ptolemy variety. An abstract triangulation can be thought of as an equivalence class of concrete triangulations under reordering. As we shall see, a reordering changes the varieties by canonical isomorphisms. Hence, they only depend on the abstract triangulation.

Note that the vertex ordering of each simplex induces an orientation, which may or may not agree with the orientation inherited from $M$.

\begin{definition}
A concrete triangulation of $M$ is an \emph{oriented triangulation} if the orientation of each simplex agrees with the orientation of $M$. A concrete triangulation of $M$ is an \emph{ordered triangulation} if the face-pairings are order-preserving. An abstract triangulation is {\it orderable} if it supports an ordered triangulation.
\end{definition}


As we shall see, the shape coordinates are most conveniently expressed in terms of \emph{oriented triangulations}, whereas the Ptolemy coordinates are most conveniently expressed in terms of \emph{ordered triangulations}. Note that since $M$ is assumed to be oriented, one can always order the vertices making the triangulation oriented.

 
\begin{remark} 
One can always obtain an orderable triangulations by performing a sequence of $2$-$3$ moves and $1$-$4$ moves. One can do this systematically in such a way that the total number of simplices is increased by at worst a factor of $6$. Alternatively, a barycentric subdivision always provides an ordered triangulation by ordering vertices by codimension.  
\end{remark}

\begin{figure}[htpb]
\centering
\begin{minipage}[c]{0.49\textwidth}
\includegraphics[width=8.4cm]{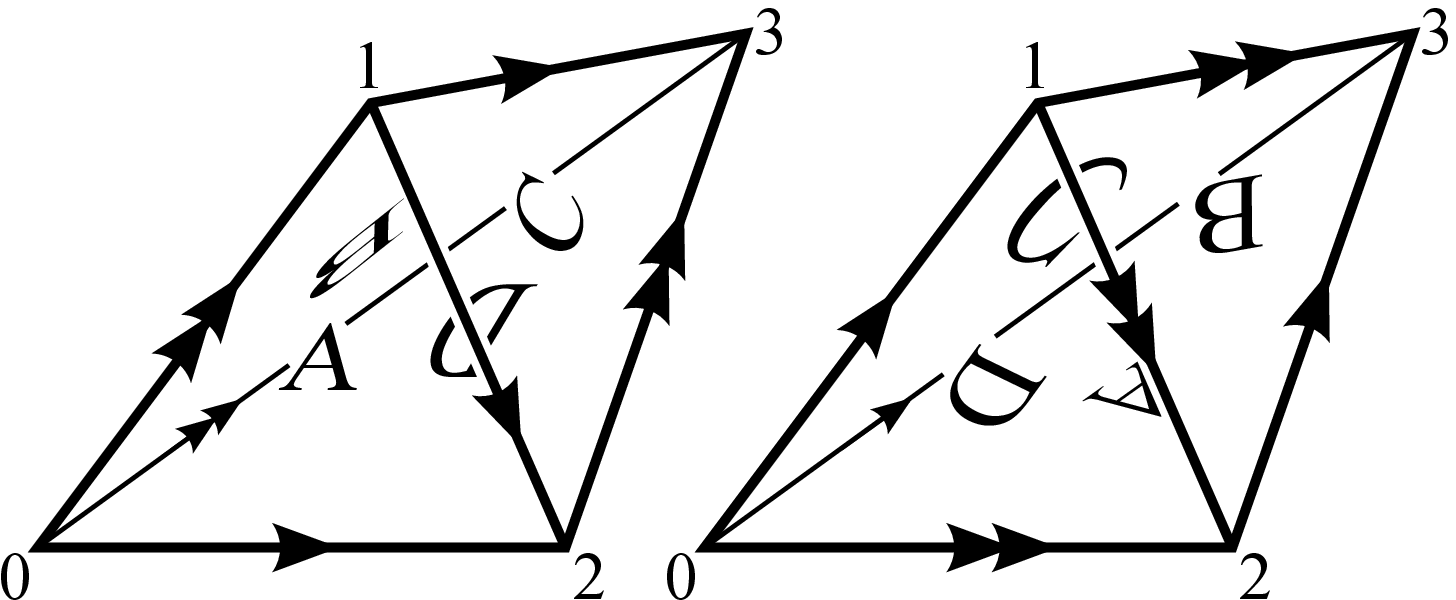}
\captionsetup{width=10cm}
\caption{An ordered, but not oriented triangulation of the figure $8$ knot. No vertex ordering exists making the triangulation both ordered and oriented.}\label{fig:Fig8Ordered}
\end{minipage}	
\begin{minipage}[c]{0.49\textwidth}
\includegraphics[width=8.4cm]{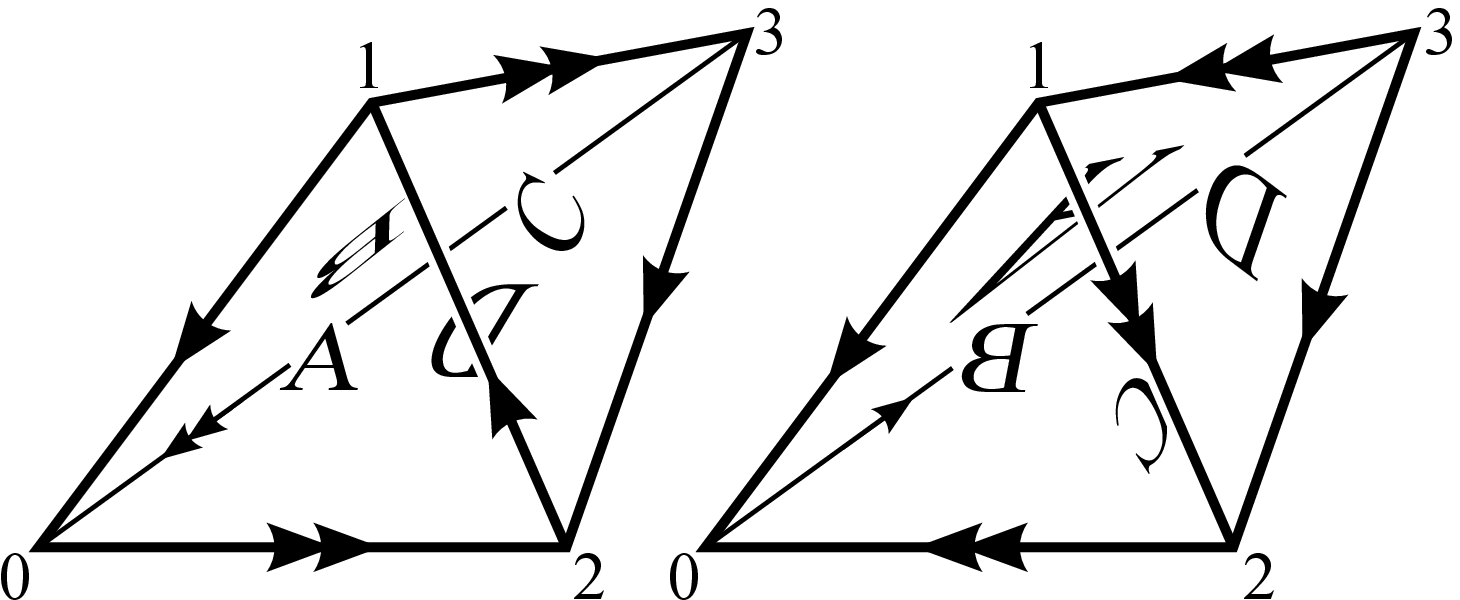}
\captionsetup{width=10cm}
\caption{An oriented, but not ordered triangulation of the figure $8$ knot sister. The underlying abstract triangulation is unorderable.}
\end{minipage}					
\end{figure}



\subsection{Face pairing permutations}
We canonically identify the symmetry group of an ordered simplex with
$S_4$.
\begin{definition} 
Let $\Delta_0$ and $\Delta_1$ be ordered simplices and let $\psi$ from face $f_0$ of $\Delta_0$ to face $f_1$ of $\Delta_1$ be a face pairing. The \emph{face pairing permutation} corresponding to $\psi$ is the unique permutation $\sigma\in S_4$ such that $\psi$ takes vertex $v$ of $\Delta_0$ to vertex $\sigma(v)$ of $\Delta_1$ whenever $v$ is a vertex in $f_0$. 
\end{definition}
Note that if we identify $\Delta_0$ and $\Delta_1$ via the unique order preserving isomorphism, $\sigma$ is the unique extension of $\psi$ to a symmetry of $\Delta_0$. See Figure~\ref{fig:faceGluing}.


\section{Thurston's gluing equations}\label{sec:ThurstonReview}
In this section we briefly review Thurston's gluing equations. For details we refer to Thurston~\cite{ThurstonNotes} or Neumann-Zagier~\cite{NeumannZagier}.

Let $\T$ be an \emph{oriented} triangulation of $M$. The gluing equations are given in terms of a variable $z_\Delta\in\C\setminus\{0,1\}$, called a \emph{shape coordinate}, for each simplex $\Delta$ of $M$. To define the equations, assign to each edge of each simplex $\Delta$ of $M$ one of three \emph{shape parameters}, see Figure~\ref{fig:TraditionalShapes}. The shape parameters are given in terms of the shape coordinate $z_\Delta$ by
\begin{equation} \label{eqn:shapeParms}
z_\Delta,\qquad z_\Delta^\prime=\frac{1}{1-z_\Delta},\qquad z_\Delta^{\prime\prime}=1-\frac{1}{z_\Delta}=-\frac{1-z_\Delta}{z_\Delta}.
\end{equation}

\begin{figure}[htpb]
\begin{center}
\includegraphics[width=4.5cm]{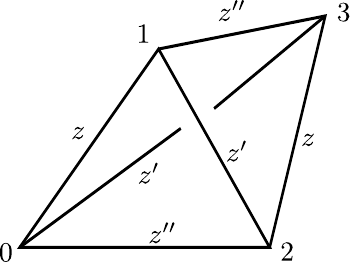}
\end{center}
\caption{Assigning shape parameters to the edges of a 
simplex.}\label{fig:TraditionalShapes}
\end{figure}

The gluing equations consist of \emph{edge equations} and \emph{cusp equations}. There is an edge equation for each edge cycle $e$ of $\mathcal T$, obtained by setting the product of the shape parameters assigned to each edge in $e$ equal to $1$. 
Each edge equation thus has the form 
\begin{equation}\label{eqn:ThurstonGluing}
\prod_\Delta z_\Delta^{a_\Delta} \prod_\Delta z'^{b_\Delta}_\Delta \prod_\Delta z''^{c_\Delta}_\Delta = 1,
\end{equation}
where $a_\Delta$, $b_\Delta$ and $c_\Delta$ are integers.

The cusp equations consist of an equation for each generator of the fundamental group of each boundary component. If $\gamma$ is a peripheral (normal) curve, we obtain a cusp equation by setting the product of the shape parameters (or their inverses) of edges passed by $\gamma$ equal to $1$ ($\gamma$ \emph{passes} an edge $E$ of a simplex $\Delta$ if it enters and exits $\Delta$ through faces intersecting in $E$). A shape parameter appears with its inverse if and only if $\gamma$ passes $e$ in a clockwise direction viewed from the cusp. The cusp equations have the same form as \eqref{eqn:ThurstonGluing}.

The following result is well known. We will generalize this to representations in $\mathrm{PGL}(n,\C)$ below.
\begin{theorem} A solution to the edge equations with all shape coordinates in $\C\setminus\{0,1\}$ uniquely (up to conjugation) determines a representation $\rho\colon\pi_1(M)\to \PGL(2,\C)$. If the solution also satisfies the cusp equations, $\rho$ is boundary-unipotent, i.e.~takes peripheral curves to a conjugate of $N$.
\end{theorem}


\subsection{Symplectic properties of the gluing equations}
It is sometimes convenient to express the gluing equations entirely in terms of the $z_\Delta$'s. Using \eqref{eqn:shapeParms}, we can rewrite \eqref{eqn:ThurstonGluing} as
\begin{equation}
\prod_\Delta z_\Delta^{A_\Delta} \prod_\Delta (1-z_\Delta)^{B_\Delta}=\pm 1.
\end{equation}
These equations are said to be of \emph{Neumann-Zagier} type. 
Each such equation gives a row vector consisting of $A_\Delta$ and $B_\Delta$'s. The resulting matrix has some symplectic properties.


\section{Generalized gluing equations}\label{sec:GeneralizedGluingEq}
In this section we define the higher dimensional analog of Thurston's edge equations. 
The generalized cusp equations will be studied in Section~\ref{sec:CuspEquations}.  

The idea is to subdivide each simplex of $M$ into overlapping subsimplices, and assign a shape coordinate to each edge of each subsimplex. When the edge midpoints of different subsimplices intersect, we obtain a gluing equation by setting the product of the respective shape parameters equal to 1. 

\subsection{Simplex coordinates}
We identify each simplex of a concrete triangulation $\T$ with the ordered simplex 
\begin{equation}
\Delta^3_n = \left\{(x_0,x_1,x_2,x_3) \in \R^4\bigm\vert 0\leq x_i \leq n, x_0 + x_1 + x_2 + x_3 = n\right\}.
\end{equation}
By removing the four vertices, we obtain the ideal standard simplex
$\dot{\Delta}^3_n$.
Consider the sets
$$\Delta^3_n(\Z)=\Delta^3_n\cap \Z^4,\qquad \dot{\Delta}^3_n(\Z)=\dot{\Delta}^3_n\cap\Z^4,\quad\mbox{and}\quad \Delta^3_n(\Z_+)=\Delta^3_n\cap\Z_+^4$$
of integral points, non-vertex integral points, and integral points lying entirely inside the simplex. A simple counting argument shows that 
\begin{equation}\label{eq:Counting}
\left|\Delta^3_n(\Z)\right| = \binom{n+3}{3},\quad \big|\dot{\Delta}^3_n(\Z)\big| = \binom{n+3}{3}-4,\quad\mbox{and}\quad \left|\Delta^3_n(\Z_+)\right| = \binom{n-1}{3}.
\end{equation}
Note that $\dot{\Delta}^3_2(\Z)$ consists of the edge midpoints of $\Delta^3_2$ and thus naturally parametrize the undirected edges.

When convenient, we abbreviate tuples by dropping the parenthesis and the commas, e.g., we write 1010 instead of $(1,0,1,0)$. Note that the indices of an edge and its opposite edge add up to $1111$. 

\begin{figure}
\begin{center}
\includegraphics{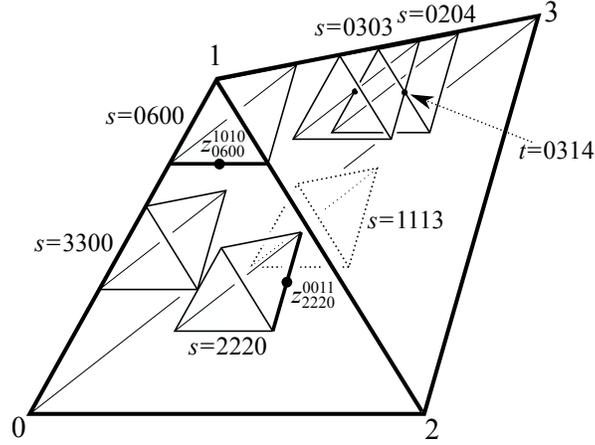}
\end{center}
\caption{Subsimplices and shape parameters for $n=8$.}
\end{figure}

\subsection{Symmetries of a simplex}
The natural vertex ordering of $\Delta_n^3$ induces an identification of the symmetry group of $\Delta^3_n$ with $S_4$, such that $\sigma\in S_4$ is the restriction to $\Delta^3_n$ of the unique linear map taking the standard basis vector $e_i$ to $e_{\sigma(i)}$, $i\in\{0,1,2,3\}$. Note that 
\begin{equation}
\sigma(x_0,x_1,x_2,x_3)=(x_{\sigma^{-1}(0)},x_{\sigma^{-1}(1)},x_{\sigma^{-1}(2)},x_{\sigma^{-1}(3)}),\qquad (x_0,x_1,x_2,x_3)\in\Delta_n^3.
\end{equation}

\subsection{Shape assignments}
We now introduce the generalized shape parameters. We will need to replace the traditional labeling of the shape parameters $z^\prime$ and $z^{\prime\prime}$ by a notation, which better exhibits the symmetry and naturally allows for a unified treatment of the gluing equations for all $n\geq 2$.

\begin{definition}\label{def:Subsimplex}
A \emph{subsimplex} of $\Delta^3_n$ is a subset $S$ of $\Delta^3_n$ obtained by translating $\Delta^3_2\subset\R^4$ by an element $s$ in $\Delta^3_{n-2}(\Z)\subset\Z^4$, i.e.~$S=s+\Delta^3_2$.
\end{definition}

Fix $n\geq 2$. 
We wish to assign shape parameters to each edge of each subsimplex. Note that the set of all these edges is naturally parameterized by the set $\Delta^3_{n-2}(\Z)\times\dot{\Delta}^3_2(\Z)$, the first coordinate being the subsimplex, and the second coordinate the edge.

\begin{definition}\label{shapeassignment}
A shape assignment on $\Delta^3_n$ is an assignment 
\begin{equation}
z\colon\Delta^3_{n-2}(\Z)\times \dot{\Delta}^3_2(\Z)\to\C\setminus\{0,1\}, \qquad (s,e)\mapsto z^e_s
\end{equation}
satisfying the \emph{shape parameter relations}
\begin{subequations}
\label{eqn:shapeparamrel}
\begin{align}
z^{0011}_s&=z^{1100}_s=\frac{1}{1-z^{0101}_s}\label{shapepar1}\\
z^{0110}_s&=z^{1001}_s=\frac{1}{1-z^{0011}_s}\label{shapepar2}\\
z^{0101}_s&=z^{1010}_s=\frac{1}{1-z^{0110}_s}\label{shapepar3}.
\end{align}
\end{subequations}
\end{definition}


\begin{remark}When $n=2$, there is only a single subsimplex indexed by $s=0000$, so $\Delta^3_{n-2}(\Z)\times\dot{\Delta}^3_2(\Z)\cong \dot{\Delta}^3_2(\Z)$ parametrizes the edges of a simplex. Since the shape parameters in \eqref{eqn:shapeParms} satisfy $$z=\frac{1}{1-z''},\qquad z'=\frac{1}{1-z},\qquad z''=\frac{1}{1-z'},$$  Definition \ref{shapeassignment} generalizes Thurston's shape assignments. The new notation relates to that of Thurston as follows:
$$z=z_s^{0011}=z_s^{1100},\qquad z^\prime=z_s^{0110}=z_s^{1001}, \qquad z''=z_s^{0101}=z_s^{1010}.$$
\end{remark}

\subsection{Gluing equations for oriented triangulations}\label{sec:GluingEquations}

Let $\T$ be an \emph{oriented} triangulation $\mathcal{T}$. The gluing equations are indexed by (non-vertex) integral points of $\mathcal{T}$ defined below and come in three flavors: \emph{edge equations}, \emph{face equations}, and \emph{interior equations}. In Section~\ref{sec:ShapeAssignmentSymmetries} we generalize the gluing equations to all concrete triangulations.

Recall that $\mathcal{T}$ is an identification of $\widehat{M}$ with a quotient of a disjoint union of standard simplices. Hence, $\mathcal{T}$ comes naturally equipped with a map
\begin{equation}
q\colon \coprod\Delta^3_n \rightarrow \widehat{M}.
\end{equation}

\begin{definition}\label{def:IntegralPoint}
An \emph{integral point} of $\mathcal{T}$ is a point $p$ in  
$$q\left(\coprod\Delta^3_n(\Z)\right)\subset \widehat{M}.$$
We view $p$ as an equivalence class of pairs $(t,\Delta)$ with $t\in\Delta^3_n(\Z)$ and $\Delta\in\mathcal T$ and write $(t,\Delta)\in p$ if $(t,\Delta)$ is a representative of $p$. The set of all integral points of $\mathcal{T}$ is denoted by $\mathcal{T}(\Z)$.
\end{definition}

\begin{definition} Let $p$ be an integral point of $\mathcal T$ represented by $(t,\Delta)$.
\begin{enumerate}[(i)]
\item We call $p$ a \emph{vertex point} if $t$ is a vertex of $\Delta^3_n(\Z)$.
\item We call $p$ an \emph{edge point} if $t$ is on an edge of $\Delta^3_n(\Z)$.
\item We call $p$ a \emph{face point} if $t$ is on a face $\Delta^3_n(\Z)$.
\item We call $p$ an \emph{interior point} if $t$ is in the interior of $\Delta^3_n(\Z)$, i.e.~if $t\in\Delta_n^3(\Z_+)$.
\end{enumerate}
We denote the set of non-vertex integral points by $\dot{\mathcal T}_n(\Z)$.
\end{definition}

\begin{definition}\label{def:GluingEquationsOri}
A shape assignment on an ordered triangulation $\mathcal{T}$ is a shape assignment $z^e_{s,\Delta}$ for each simplex $\Delta\in\mathcal{T}$ such that for each non-vertex integral point $p\in\dot{\mathcal{T}}_n(\Z)$, the \emph{generalized gluing equation}
\begin{equation}
\prod\limits_{(t,\Delta)\in p}\,\prod\limits_{t = s + e} z^e_{s,\Delta} = 1. \label{eqn:GeneralizedGluing}
\end{equation}
is satisfied. The variety of shape assignments on $\T$ is denoted by $V_n(\T)$.
\end{definition}

\begin{figure}
\begin{center}
\subfigure[$z_0z''_1z'_2=1$]{
\includegraphics[scale=0.35]{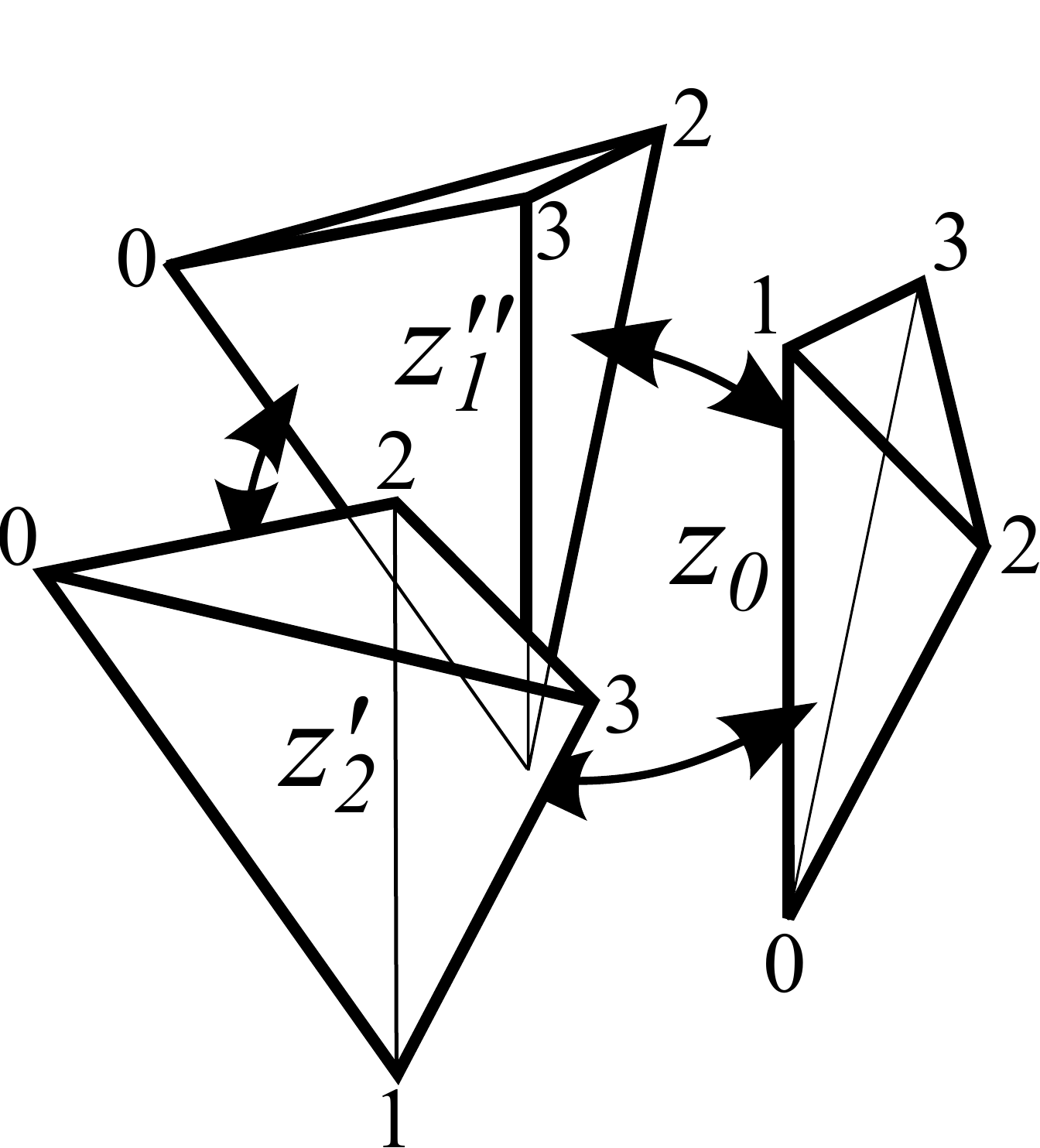}} 
\subfigure[$z_{1200,0}^{1100}z_{0102,1}^{0101}z_{0120,2}^{0110}=1$]{
\includegraphics[scale=0.35]{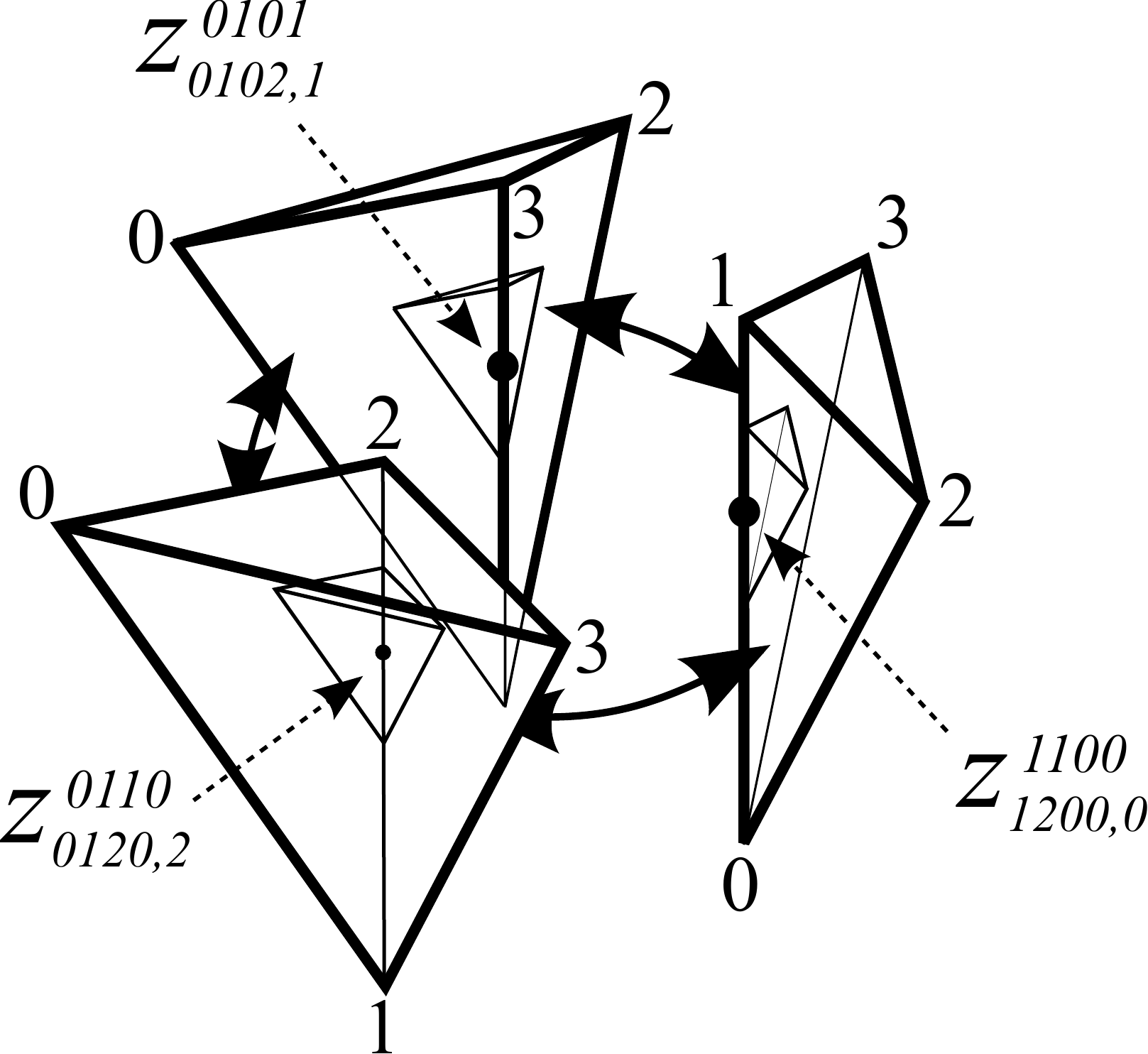}}
\end{center}
\caption{Edge equations for $n=2$ (traditional notation) and $n=5$. \label{fig:edgeGluing}}
\end{figure}

\begin{figure}
\begin{center}
\scalebox{0.9}{\input{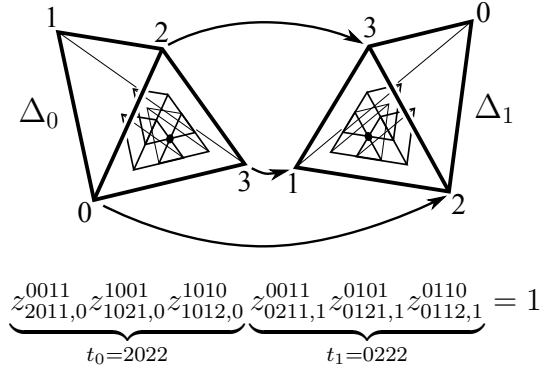}}\\
\bigskip
$\phantom{=1} \underbrace{z^{0011}_{2011,0}z^{1001}_{1021,0}z^{1010}_{1012,0}}_{t_0=2022}\underbrace{z^{0011}_{0211,1}z^{0101}_{0121,1}z^{0110}_{0112,1}}_{t_1=0222}=1$
\end{center}
\caption{A face equation for $n=6$. The indicated face-pairing is encoded by the permutation $\sigma=(0231)\in S_4$.}\label{fig:faceGluing}
\end{figure}

\begin{figure}
\begin{center}
\includegraphics[scale=0.35]{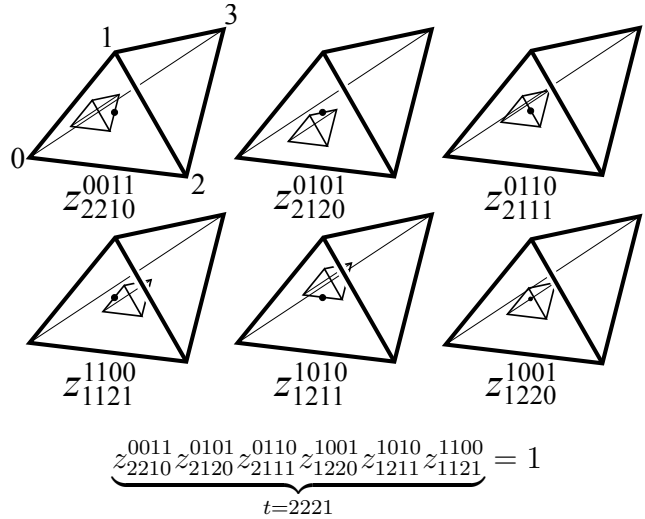}\\
\bigskip
$\underbrace{z^{0011}_{2210} z^{0101}_{2120} z^{0110}_{2111}z^{1001}_{1220} z^{1010}_{1211} z^{1100}_{1121}}_{t=2221}=1$
\end{center}
\caption{An internal equation for $n=7$.\label{fig:internalGluing}}
\end{figure}

 Note that the gluing equation for an integral point $p$ sets equal to $1$ the product of the shape parameters of all edges of subsimplices such that the edge midpoint intersects $p$.
The generalized gluing equations come in three different flavors depending on the type of the integral point $p$.
\begin{itemize}
\item {\bf Edge equations:} If $p$ is an edge point, the equation is similar to the usual gluing equation in that the number of terms equals the length of the edge cycle. There are $n-1$ edge gluing equations per edge of $\mathcal{T}$ involving shape parameters at different levels. See Figure \ref{fig:edgeGluing}.
\item {\bf Face equations:} If $p$ is a face point, the product consists of six terms with three terms from each of the two simplices sharing the face. There are $\binom{n-1}{2}$ equations per face. See Figure \ref{fig:faceGluing}.

\item {\bf Internal equations:} If $p$ is an interior point, the equation consists of six terms involving subsimplices of the same simplex, i.e.~this equation is independent of the triangulation. There are $\binom{n-1}{3}$ equations per simplex. See Figure \ref{fig:internalGluing}.
\end{itemize}

\begin{remark}
There are no vertex equations. If $p$ is a vertex point, \eqref{eqn:GeneralizedGluing} is tautologically satisfied since the product is empty. 
\end{remark}



\subsection{The pullback of a shape assignment under a symmetry}

Recall that we identify the symmetry group of $\Delta_n$ with $S_4$.

\begin{definition}\label{def:PullbackShapes}
Let $\sigma\in S_4$ and let $z\colon\Delta^3_{n-2}(\Z)\times\dot\Delta^3_2(\Z)\to \C\setminus\{0,1\}$ be a map. 
The \emph{pullback} of $z$ under $\sigma$ is the map given by
\begin{equation}\label{actiononshapeassignments}
\sigma^*z\colon \Delta^3_{n-2}(\Z)\times\dot\Delta^3_2(\Z)\to\C\setminus\{0,1\},\qquad (s,e)\mapsto (z^{\sigma(e)}_{\sigma (s)})^{\sgn(\sigma)}.
\end{equation}
\end{definition}

The pullback obviously satisfies the standard properties $\tau^*\sigma^*=(\sigma\tau)^*$ and $\id^*=\id$.

\begin{lemma}\label{shapesymmetries}
The pullback preserves shape assignments.
\end{lemma}

\begin{proof}
Let $z$ be a shape assignment. One easily checks that \eqref{eqn:shapeparamrel} is preserved under the $A_4$ action on the indices, so the result follows for $\sigma\in A_4$. Hence, all that remains is to prove the result for the permutation $\sigma_{01}$ switching $0$ and $1$. The equation
\begin{equation}
(\sigma_{01}^*z)_s^{0011}=(\sigma_{01}^*z)_s^{1100}=\frac{1}{1-(\sigma_{01}^*z)_s^{0101}}
\end{equation}
is equivalent to
\begin{equation}
(z^{0011}_{\sigma_{01}(s)})^{-1}=(z^{1100}_{\sigma_{01}(s)})^{-1}=\frac{1}{1-(z_{\sigma_{01}(s)}^{1001})^{-1}},
\end{equation}
which follows from \eqref{shapepar1} and \eqref{shapepar2}. The other equations are similarly verified.
\end{proof}

Note that if $\sigma\in A_4$ is a rotation, the pullback $\sigma^* z$ is the shape assignment obtained from $z$ by rotating the simplex by $\sigma$. If $\sigma$ is orientation reversing, one must also replace all the shape parameters by their inverses.

\begin{remark}\label{rm:reordering}
We view the pullback of a shape assignment on an ordered simplex $\Delta$ as the natural induced shape assignment on the simplex $\Delta'$ obtained from $\Delta$ by reordering the vertices such that the $i$th vertex of $\Delta'$ is the $\sigma(i)$th vertex of $\Delta$.
\end{remark}

\begin{figure}[htb]
\begin{center}
\begin{minipage}[b]{0.45\textwidth}
\input{figures_gen/ActionReorder.tex} 
\end{minipage}
\hfill
\begin{minipage}[b]{0.45\textwidth}
\input{figures_gen/ActionPullback.tex}
\end{minipage}
\\
\begin{minipage}[t]{0.45\textwidth}
\caption{Reordering by $\sigma=(123)$.}
\end{minipage}
\begin{minipage}[t]{0.45\textwidth}
\caption{Pullback by $\sigma=(123)$.}
\end{minipage}
\end{center} 
\end{figure}

\subsection{Gluing equations for general concrete triangulations}\label{sec:ShapeAssignmentSymmetries}

Since every concrete triangulation can be obtained from an oriented one by reordering some of the vertices, Lemma~\ref{shapesymmetries} motivates the following.

\begin{definition}\label{def:ShapeAssignmentGeneral} A shape assignment on a concrete triangulation $\T$ is a shape assignment on each simplex such that
\begin{equation}\label{eqn:GeneralizedGluingOrderAgnostic}
\prod\limits_{(t,\Delta)\in p}\,\prod\limits_{t = s + e} (z^e_{s,\Delta})^{\epsilon_\Delta} = 1, 
\end{equation}
where $\epsilon_\Delta$ is a sign indicating whether or not the orientation of $\Delta$ given by the vertex ordering agrees with the orientation inherited from $M$. The variety of shape assignments is denoted by $V_n(\T)$.
\end{definition}

The following is an immediate corollary of Lemma~\ref{shapesymmetries}, c.f.~Remark~\ref{rm:reordering}.
\begin{lemma}\label{lemma:reorderingshapes}
Let $\{z_{\Delta_i}\}$ be a shape assignment on $(M,\T)$ and let $\T'$ be the triangulation obtained from $\T$ by reordering each simplex $\Delta_i$ by a permutation $\sigma_i\in S_4$. The shape assignments $\{\sigma_i^*z_{\Delta_i}\}$ form a shape assignment on $(M,\T')$. \qed
\end{lemma}

\begin{corollary}
Up to canonical isomorphism, the gluing equation variety only depends on the abstract triangulation.\qed
\end{corollary}

\begin{remark}
The gluing equations can also be defined if $M$ is non-orientable: pick an oriented neighborhood $U$ of the integer point $p$. The sign $\epsilon_\Delta$ now indicates whether or not the orientation of $\Delta$ agrees with the orientation of $U$. We will not explore this further.
\end{remark}

\section{Review of Ptolemy coordinates}\label{sec:PtolemyReview}

Ptolemy coordinates were introduced in \cite{GaroufalidisThurstonZickert} inspired by $\mathcal A$-coordinates on higher Teichm\"{u}ller spaces due to Fock and Goncharov~\cite{FockGoncharov}. They are indexed by non-vertex integral points of $\mathcal{T}$ satisfying \emph{Ptolemy relations} each involving the six Ptolemy coordinates assigned to the edges of a subsimplex. They are most naturally defined for ordered triangulations. General concrete triangulations are studied in Section~\ref{sec:OrderAgnosticPtolemy}.

\subsection{Ptolemy assignments for ordered triangulations}
\begin{definition}\label{def:PtolemyAssignment}
A \emph{Ptolemy assignment} on $\Delta^3_n$ is an assignment
\begin{equation}
\dot{\Delta}^3_n(\Z)\to\C\setminus\{0\},\qquad t\mapsto c_t
\end{equation}
of a non-zero complex number $c_t$ to each non-vertex integral point $t$ of $\Delta^3_n$ such that the \emph{Ptolemy relation}
\begin{equation}
c_{s+1001}c_{s+0110}+c_{s+1100}c_{s+0011}=c_{s+1010}c_{s+0101}\label{eqn:PtolemyRelation}
\end{equation}
is satisfied for each subsimplex $s\in\Delta^3_{n-2}(\Z)$. 
\end{definition}
\begin{definition}\label{def:PtolemyAssignmentOrd}
A Ptolemy assignment on an \emph{ordered} triangulation $\mathcal{T}$ is an assignment
\begin{equation}
\dot{\mathcal{T}}_n(\Z)\to\C\setminus\{0\},\qquad p\mapsto c_p
\end{equation}
of a non-zero complex number to each non-vertex integral point $p$ of $\mathcal{T}$ such that for each simplex in $\mathcal{T}$ the identification with $\Delta^3_n$ induces a Ptolemy assignment on $\Delta^3_n$. If $(t,\Delta)\in p$ is a representative of $p$, we write $c_{t,\Delta}$ for the Ptolemy coordinate $c_p$. The variety of Ptolemy assignments is denoted by $P_n(\T)$.
\end{definition}

\begin{remark}
Whenever convenient, we extend a Ptolemy assignment, so that it takes vertex points to $1$.
\end{remark}

\begin{figure}[htb]
\begin{center}
\scalebox{0.9}{
\includegraphics{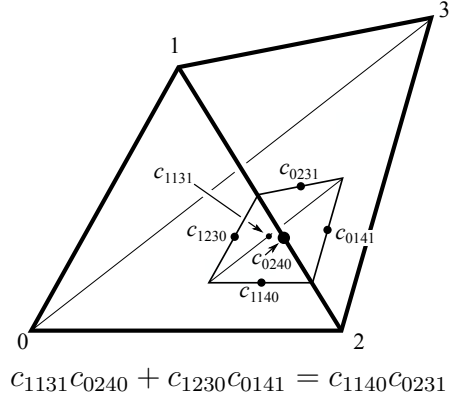}}\\
$c_{1131}c_{0240}+c_{1230}c_{0141}=c_{1140}c_{0231}$
\end{center}
\caption{Ptolemy relation for the subsimplex $s=0130$.}
\end{figure}

\begin{remark} Note that the Ptolemy relation \eqref{eqn:PtolemyRelation} is local, i.e.~independent of the triangulation $\mathcal{T}$. The triangulation determines whether $(t,\Delta)$ and $(t',\Delta')$ represent the same integral point $p$, and, hence, whether the Ptolemy coordinates $c_{t,\Delta}$ and $c_{t',\Delta'}$ are identified.
\end{remark}

\subsection{The pullback of a Ptolemy assignment under a symmetry} 
The Ptolemy coordinates are not as well behaved under symmetries as the shapes. 
The obvious pullback defined by $(\sigma^*c_t)=c_{\sigma(t)}$ does not preserve Ptolemy assignments. To fix this, we must modify by signs depending on both $\sigma\in S_4$ and $t$.

Let $I$ denote the identity matrix in $\GL(n,\C)$. For each $t\in\Delta_n^3(\Z)$, we can write $I$ as a concatenation of $n\times t_i$ matrices $I_i^t$, i.e.~we have $I=[I_0^t\vert I_1^t\vert I_2^t\vert I_3^t]$. For $\sigma\in S_4$ define
\begin{equation}
I_{\sigma,t}=[I_{\sigma(0)}^t\vert I_{\sigma(1)}^t\vert I_{\sigma(2)}^t\vert I_{\sigma(3)}^t].
\end{equation}
Note that 
\begin{equation} \label{Isigmaproperties}
I_{\sigma \tau,\sigma \tau(t)} =I_{\sigma, \sigma\tau (t)} I_{\tau,\tau(t)}, \qquad \det(I_{\sigma_{ij},t}) = \det(I_{\sigma_{ij},\sigma_{ij}(t)})=(-1)^{t_it_j} ,
\end{equation}
where $\sigma_{ij}$ is the permutation switching $i$ and $j$.

\begin{definition}\label{def:PullbackPtolemy} 
Let $\sigma\in S_4$ and let $c\colon\dot\Delta^3_n(\Z)\to\C^*$ be a map. The \emph{pullback} of $c$ under $\sigma$ is the map
\begin{equation}\label{ActiononPtolemys}
(\sigma^* c)_t=\det(I_{\sigma,\sigma(t)})c_{\sigma(t)}.
\end{equation}
\end{definition}

Using~\eqref{Isigmaproperties}, one checks that the pullback satisfies the properties $\tau^*\sigma^*=(\sigma\tau)^*$ and $\id^*=\id$.

\begin{remark}
The formula is motivated by Lemma~\ref{lemma:DecorationsAndPullbacks} below.
\end{remark}

\begin{remark}\label{rm:ShuffleOdd}
Note that  $\det(I_{\sigma,\sigma(t)})$ only depends on the parity of the entries of $t$ (and on $\sigma$). It equals the sign of the permutation shuffling the odd entries of $t$, e.g., if $\sigma$ takes $t_0=(0,0,3,1)$ to $t_1=(0,1,0,3)$,  $\det(I_{\sigma,\sigma(t_0)})=-1$ since the permutation taking $(3,1)$ to $(1,3)$ is odd. See, e.g.,  Figure~\ref{fig:PtolemyIdentification}.
\end{remark}

\begin{lemma}\label{lemma:PtolemySymmetries}
The pullback preserves Ptolemy assignments.
\end{lemma}

\begin{proof}
Let $c\colon\dot\Delta^3_n\to\C^*$ be a Ptolemy assignment. Since $S_4$ is generated by the transpositions $\sigma_{01}$, $\sigma_{12}$ and $\sigma_{23}$, it is enough to prove the result for these. We prove it for $\sigma_{01}$, the others being similar. We wish to prove that
\begin{equation}\label{eq}
(\sigma_{01}^* c)_{s+1001}(\sigma_{01}^* c)_{s+0110}+(\sigma_{01}^* c)_{s+1100}(\sigma_{01}^* c)_{s+0011}-(\sigma_{01}^* c)_{s+1010}(\sigma_{01}^* c)_{s+0101}=0.
\end{equation}
Using, \eqref{ActiononPtolemys} and \eqref{Isigmaproperties} and letting $s'=\sigma_{01}s$, the left side of \eqref{eq} becomes
\begin{equation}\setlength\arraycolsep{0.01em}\label{eq:PullbackPtolemy}
\begin{array}{llll}
&(-1)^{(s_0+1)s_1} & (-1)^{s_0(s_1+1)} & c_{s'+0101}c_{s'+1010}\\
+&(-1)^{(s_0+1)(s_1+1)} & (-1)^{s_0s_1} & c_{s'+1100} c_{s'+0011}\\
-&(-1)^{s_0(s_1+1)} & (-1)^{s_1(s_0+1)} & c_{s'+1001}c_{s'+0110}\\
= && (-1)^{s_0+s_1} & \left(c_{s'+0101}c_{s'+1010}-c_{s'+1100} c_{s'+0011}-c_{s'+1001}c_{s'+0110}\right).
\end{array}
\end{equation}
By the Ptolemy relation for $s'$ this equals $0$, proving the result.
\end{proof}

As in Remark~\ref{rm:reordering}, we shall view the pullback as a natural induced Ptolemy assignment on a reordered simplex.

\subsection{Ptolemy assignments for general concrete triangulations}
\label{sec:OrderAgnosticPtolemy}

For general concrete triangulations the Ptolemy coordinates on faces of different simplices must be identified by signs given by the face pairing permutations.
\begin{definition}\label{def:PtolemyAssignmentGeneral}
A Ptolemy assignment on $(M,\T)$ is a Ptolemy assignment for each simplex of $\T$ such that Ptolemy coordinates on identified faces are identified via the pullback of the permutation matrix. More precisely, if $f_0\subset\Delta_0$ is paired with $f_1\subset\Delta_1$ via the permutation $\sigma$, we require that 
\begin{equation}
(\sigma^*c_{\Delta_1})_{t_0}= (c_{\Delta_0})_{t_0}
\end{equation}
for each $t_0\in\Delta_n^3$ on face $f_0$. Equivalently, we require that $(c_{\Delta_0})_{t_0} = \det(I_{\sigma,\sigma(t_0)}) (c_{\Delta_1})_{\sigma(t_0)}$.
The variety of Ptolemy assignments is denoted by $P_n(\T)$.
\end{definition}

\begin{figure}[htb]
\begin{center}
\input{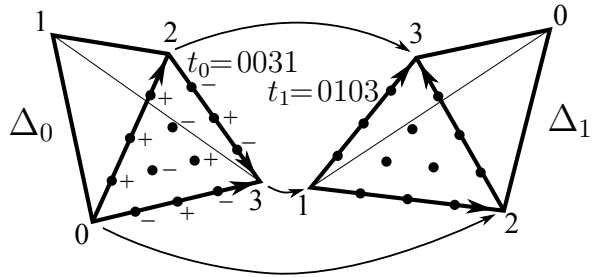} 
\end{center}
\caption{Identification of Ptolemy coordinates.}\label{fig:PtolemyIdentification}
\end{figure}

Note that if $\T$ is an ordered triangulation, all face pairings are order preserving, so all signs are positive, and the definition agrees with Definition~\ref{def:PtolemyAssignment}.

\begin{lemma}\label{lemma:reorderingPtolemys}
Let $\T'$ be the triangulation obtained from $\T$ by reordering the vertices of $\Delta_i$ by a permutation $\sigma_i$.
Then $\{\sigma_i^*c_{\Delta_i}\}$ is a Ptolemy cochain on $\T'$.
\end{lemma}

\begin{proof} We must prove that the coordinates on identified faces are identified via the pullback.
Let $c_{\Delta_i'}=\sigma_i^*c_{\Delta_i}$. Suppose $f_i\subset\Delta_i$ is glued to $f_j\subset\Delta_j$ by a permutation $\tau$. The corresponding face pairing involving $\Delta_i'$ and $\Delta_j'$ is then $\tau'=\sigma_j^{-1}\tau\sigma_i$. Since $\tau^*c_{\Delta_j}$ agrees with $c_{\Delta_i}$ on face $f_0$, it follows from the standard pullback properties that $\tau'^*c_{\Delta_j'}$ and $c_{\Delta_i'}$ also agree on $f_0$. This proves the result.
\end{proof}
\begin{corollary}
Up to canonical isomorphism, the Ptolemy variety only depends on the abstract triangulation.\qed
\end{corollary}

\section{From Ptolemy coordinates to shapes}\label{sec:PtolemyToShapes}

We now define a monomial map $\mu$ from Ptolemy assignments to shape assignments.
Given a Ptolemy assignment $c$ on a simplex $\Delta^3_n$, define $z\colon\Delta^3_{n-2}(\Z)\times\dot\Delta_2^3(\Z)$ by
\begin{equation}\label{eqn:PtolemyToShapes}
\begin{aligned}
z_s^{1100}  = z_s^{0011} & = & & \frac{c_{s+1001}c_{s+0110}}{c_{s+1010}c_{s+0101}}\\
z_s^{0110}  = z_s^{1001} & =  & & \frac{c_{s+0101}c_{s+1010}}{c_{s+1100}c_{s+0011}}\\
z_s^{1010}  = z_s^{0101} & =  & - & \frac{c_{s+1100}c_{s+0011}}{c_{s+1001}c_{s+0110}}.
\end{aligned}
\end{equation}

\begin{lemma}
The assignment \eqref{eqn:PtolemyToShapes} is a shape assignment, i.e.~we have a well defined map 
\begin{equation}\mu\colon\text{\{Ptolemy assignments on $\Delta^3_n$\}}\to \text{\{shape assignments on $\Delta^3_n$\}}.
\end{equation}
\end{lemma}

\begin{proof}
Using the Ptolemy relation, we obtain
$$\frac{1}{1-z_s^{0101}}=\frac{1}{1+\frac{c_{s+1100}c_{s+0011}}{c_{s+1001}c_{s+0110}}}=\frac{c_{s+1001}c_{s+0110}}{c_{s+1100}c_{s+0011}+c_{s+1001}c_{s+0110}}=\frac{c_{s+1001}c_{s+0110}}{c_{s+1010}c_{s+0101}}=z_s^{1100}.$$
The other two equations in \eqref{eqn:shapeparamrel} follow similarly.
\end{proof}

\begin{lemma}\label{lemma:ReorderCompatibility}
The map $\mu$ respects pullbacks.
\end{lemma}
\begin{proof}
It is enough to prove this for the permutations $\sigma_{01}$, $\sigma_{12}$ and $\sigma_{23}$. We prove it for $\sigma_{01}$, the other cases being similar. Since $\sgn(\sigma_{01})=-1$, we must prove that 
\begin{equation}\label{eqn:murespectspullback}(\mu(\sigma_{01}^*c))^e_s=(\mu(c)_{\sigma_{01}(s)}^{\sigma_{01}(e)})^{-1}.\end{equation}
We prove this for the edge $e=1100$, the other cases being similar. Using \eqref{eq:PullbackPtolemy}, we obtain
\begin{equation}
(\mu(\sigma_{01}^*c))^{1100}_s=\frac{(\sigma_{01}^*c)_{s+1001}(\sigma_{01}^*c)_{s+0110}}{(\sigma_{01}^*c)_{s+1010})(\sigma_{01}^*c)_{s+0101}}=\frac{c_{s'+0101}c_{s'+1010}}{c_{s'+0110}c_{s'+1001}}=(\mu(c)_{s'}^{1100})^{-1}
\end{equation}
where $s'=\sigma_{01}(s)$. This proves \eqref{eqn:murespectspullback}, hence the result.
\end{proof}

\begin{theorem}\label{thm:PtolemyYieldsShapes}
Let $\mathcal{T}$ be a concrete triangulation of $M$, and let $c_{t,\Delta}$ be a Ptolemy assignment on $(M,\mathcal{T})$. The induced shape assignment on each simplex satisfies the generalized gluing equations and thus induces a shape assignment on $(M,\mathcal{T})$. \qed
\end{theorem}

\begin{corollary} The map $\mu$ induces a map $\mu\colon P_n(\T)\to V_n(\T)$.\qed
\end{corollary}


We divide the proof of Theorem~\ref{thm:PtolemyYieldsShapes} into three parts, one for each type of equation. The idea is to prove the result for the local model near an integral point of each type. Although the local model is not a manifold, Ptolemy assignments are defined in the obvious way, i.e.~by identifying Ptolemy coordinates of identified faces via the pullback of the face pairing permutations.
\subsection{Proof for edge equations}\label{subsection:ProofForEdgeEqn}
Let $K$ be the space defined by cyclically gluing together $k$ simplices $\Delta_0,\dots,\Delta_{k-1}$ along the common edge $01$, pairing all faces via the permutation $(23)$, see Figure~\ref{fig:EdgeLocal}. 

\begin{figure}[htb]
\begin{center}
\begin{minipage}[b]{0.45\textwidth}
\begin{center}
\scalebox{0.75}{
\includegraphics{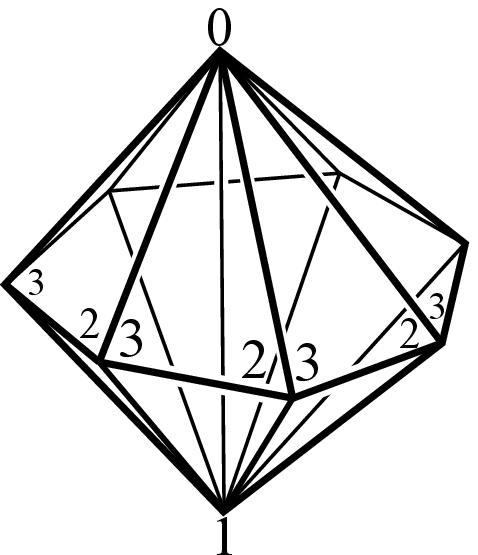}}
\end{center}
\end{minipage}
\hfill
\begin{minipage}[b]{0.45\textwidth}
\begin{center}
\scalebox{0.5}{
\includegraphics{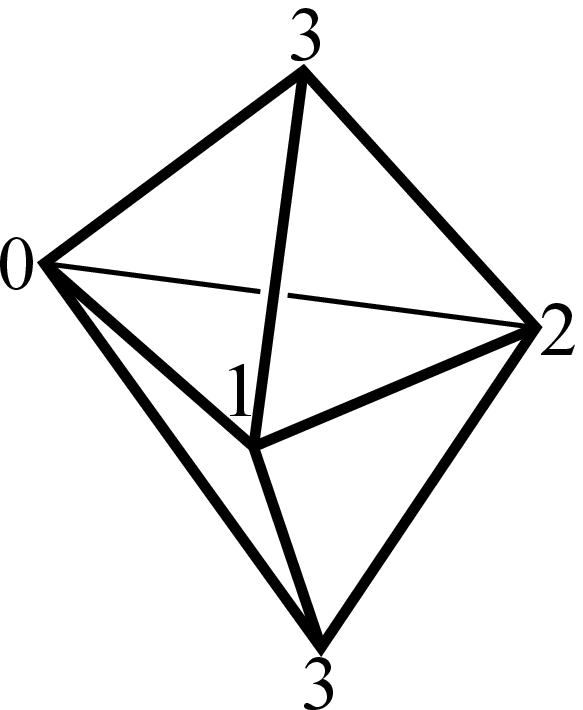}}
\end{center}
\end{minipage}
\\
\begin{minipage}[t]{0.45\textwidth}
\caption{Local model near an edge point.}\label{fig:EdgeLocal}
\end{minipage}
\hfill
\begin{minipage}[t]{0.45\textwidth}
\caption{Local model near a face point.}\label{fig:FaceLocal}
\end{minipage}
\end{center}
\end{figure}

\begin{lemma}\label{lemma:edgelocalmodel}
If $\{c_{\Delta_i}\}$ is a Ptolemy assignment on $K$, the assignments $\{\mu(c_{\Delta_i})\}$ satisfy the edge equations for all integral points on the edge $01$. The same holds for the complex $K'$ obtained by reordering the vertices of the simplices of $K$.
\end{lemma}
\begin{proof}
An integral point $p$ on $01$ has representatives $(t,\Delta_i)$ with $t=(t_0,t_1,0,0)$ being fixed. The left hand side of the edge equation for $p$ is given by
\begin{equation}\label{eq:edgeequationforK}
z^{1100}_{s,0}z^{1100}_{s,1}\cdots z^{1100}_{s,k-1}\qquad\mbox{with}~s=(t_0-1,t_1-1,0,0),
\end{equation}
which expands to
$$\frac{c_{s+1001,0}c_{s+0110,0}}{c_{s+1010,0}c_{s+0101,0}}\frac{c_{s+1001,1}c_{s+0110,1}}{c_{s+1010,1}c_{s+0101,1}}\cdots \frac{c_{s+1001,k-1}c_{s+0110,k-1}}{c_{s+1010,k-1}c_{s+0101,k-1}}.$$
Since the Ptolemy coordinates $c_{1001,i}=c_{1010,i+1}$ and $c_{0101,i}=c_{0110,i+1}$ on two adjacent simplices are identified, all terms cancel. Hence, \eqref{eq:edgeequationforK} equals $1$ as desired. The corresponding result for $K'$ follows from compatibility under reordering (Lemmas~\ref{lemma:ReorderCompatibility}, \ref{lemma:reorderingshapes} and \ref{lemma:reorderingPtolemys}).
\end{proof}

\begin{corollary} Theorem~\ref{thm:PtolemyYieldsShapes} holds for the edge equations.
\end{corollary}
\begin{proof}
The complex $K$ models a neighborhood around an edge $e$ of $\T$ in the sense that there is a simplical map $\pi\colon K\to\widehat{M}$ mapping $01$ to $e$, which is unique up to changing the orientation of $K$ and cyclically relabeling the simplices. By compatibility under reordering, we may assume that $\pi$ is order preserving. A Ptolemy assignment on $\T$ pulls back to a Ptolemy assignment on $K$, such that edge equations on $01$ descend to the corresponding edge equations on $e$. The result now follows from Lemma~\ref{lemma:edgelocalmodel}.  
\end{proof}

\subsection{Proof for face equations}\label{subsection:ProofForFaceEqn}
A local model for a neighborhood of a face point $p$ is the complex $K$ obtained by gluing two simplices $\Delta_0$ and $\Delta_1$ by identifying the faces 012 of each simplex in the order-preserving way, see Figure~\ref{fig:FaceLocal}. As for the edge equations, it is enough to verify the face equations on $K$. 
\begin{lemma}
A Ptolemy assigment on $K$ gives shape assignments satisfying the face equations. The same holds after reordering.
\end{lemma}
\begin{proof}
Let $t=(t_0,t_1,t_2,0)=\alpha+1110$ where $\alpha$ is a face point of $\Delta^3_{n-3}$. The representatives of the corresponding face point $p$ of $K$ are $(t,\Delta_0)$ and $(t,\Delta_1)$. Since $\Delta_0$ and $\Delta_1$ have opposite orientations in $K$, the face equation for $p$ involves the terms
\begin{equation}\label{eqn:proofFaceEqn}
z^{0110}_{\alpha+1000,0}z^{1010}_{\alpha+0100,0}z^{1100}_{\alpha+0010,0} \left(z^{0110}_{\alpha+1000,1} z^{1010}_{\alpha+0100,1} z^{1100}_{\alpha+0010,1}\right)^{-1}.
\end{equation}
Using \eqref{eqn:PtolemyToShapes}, the product of the first three terms equals
\begin{equation}\frac{c_{\alpha+2010,0}c_{\alpha+1101,0}}{c_{\alpha+2100,0}c_{\alpha+1011,0}} \left(-\frac{c_{\alpha+1200,0}c_{\alpha+0111,0}}{c_{\alpha+1101,0}c_{\alpha+0210,0}}\right) \frac{c_{\alpha+1011,0}c_{\alpha+0120,0}}{c_{\alpha+1020,0}c_{\alpha+0111,0}},
\end{equation}
which simplifies to
\begin{equation}\label{eq:Xcoordinate}
-\frac{c_{\alpha+2010,0} c_{\alpha+1200,0}c_{\alpha+0120,0}}{c_{\alpha+2100,0} c_{\alpha+0210,0} c_{\alpha+1020,0}}.
\end{equation}
Note that the ratio~\eqref{eq:Xcoordinate} only involves Ptolemy coordinates on the face 012 of $\Delta_0$ and that these are identified with the corresponding Ptolemy coordinates on $\Delta_1$. Hence, \eqref{eq:Xcoordinate} equals the corresponding expression for $z^{0110}_{\alpha+1000,1} z^{1010}_{\alpha+0100,1} z^{1100}_{\alpha+0010,1}$, so \eqref{eqn:proofFaceEqn} equals 1 as desired. The second statement follows from compatibility under reordering.
\end{proof}
\begin{corollary}
Theorem~\ref{thm:PtolemyYieldsShapes} holds for the face equations.\qed
\end{corollary}

\begin{remark}
The ratio \eqref{eq:Xcoordinate} will reappear in later sections as $X$-coordinates. They agree with the $X$-coordinates considered by Fock and Goncharov~\cite{FockGoncharov}. 
\end{remark}

\subsection{Proof for internal equations} \label{subsection:ProofForInternalEqn}
A local model near an interior point is a single simplex.
\begin{lemma}\label{lemma:InternalEquation} A Ptolemy assignment on $\Delta^3_n$ gives rise to a shape assignment satisfying the internal gluing equations.
\end{lemma}
\begin{proof}
Let $t$ be of the form $t=\alpha+1111$ with $\alpha\in\Delta^{3}_{n-4}(\Z)$. The gluing equation for $t$ involves
\begin{equation}
z^{0011}_{\alpha+1100}  z^{0101}_{\alpha+1010}    z^{0110}_{\alpha+1001}    z^{1001}_{\alpha+0110}     z^{1010}_{\alpha+0101}    z^{1100}_{\alpha+0011}.
\end{equation}
When expanding this using \eqref{eqn:PtolemyToShapes}, the two signs cancel and the numerator and denominator both consist of all Ptolemy coordinates $c_{\alpha+\beta}$ where $\beta$ is a permutation of $(0,1,1,2)$. Hence, the product is 1.
\end{proof}
\begin{corollary}
Theorem~\ref{thm:PtolemyYieldsShapes} holds for the internal gluing equations.\qed
\end{corollary}

\section{Symplectic properties and quantum topology}\label{sec:Symplectic}
In this section we prove Theorem~\ref{thm:soft}. This is done by generalizing some of the combinatorial properties of triangulations studied by Neumann~\cite{NeumannComb}. The first part follows immediately from Proposition~\ref{prop:BetaStarBeta} below, and the second part is an elementary counting argument. We assume for simplicity that the triangulation $\T$ is oriented. 

Letting
\begin{equation}
z_s=z_s^{1100}=z_s^{0011},\quad z_s'=z_s^{0110}=z_s^{1001},\quad z_s''=z_s^{1010}=z_s^{0101},
\end{equation}
it follows immediately from Definition~\ref{def:GluingEquationsOri} that the gluing equations can be written as 
\begin{equation}\label{eq:Unsimplified}
\prod z_s^{A'_{p,s}}\prod (z_s')^{B_{p,s}'}\prod (z_s'')^{C_{p,s}'}=1
\end{equation}
for integral matrices $A'$, $B'$ and $C'$, whose rows are parametrized by the integral points of $\T$ and whose columns are parametrized by the subsimplices of $\T$.

\begin{lemma}\label{lemma:Even}
For each integral point $p$, the integer $\sum_s C_{p,s}'$ is even.
\end{lemma}
\begin{proof}
This is obvious for the face equations and interior equations. Let $K$ be the local model (see Section~\ref{subsection:ProofForEdgeEqn}) near an edge point $p$, and let $e$ be the interior edge of $K$. We must prove that $e$ is a $1010$ edge or a $0101$ edge for an even number of simplices of $K$. Consider a curve $\gamma$ encircling the interior edge $e$ of $K$. The vertex ordering induces an orientation on each face of each simplex of $K$, such that when $\gamma$ passes through two faces of a simplex in $K$, the two orientations agree unless $e$ is a $1010$ edge or a $0101$ edge. Since $K$ is orientable, if follows that the number of such edges is even. This proves the result.
\end{proof}

Since $z_s'=\frac{1}{1-z_s}$ and $z_s''=-\frac{1-z}{z}$, it follows from Lemma~\ref{lemma:Even} that we can write the gluing equations as
\begin{equation}\label{eq:Simplified}
\prod z_s^{A_{p,s}}\prod (1-z_s)^{B_{p,s}}=1,
\end{equation}
where $A=A'-C'$ and $B=C'-B'$. We wish to prove that the rows of $(A\vert B)$ Poisson commute.

Recall that $\dot\Delta_2^3$ parametrizes the edges of $\Delta^3_2$. Let $J_{\Delta^3_2}$ be the abelian group generated by $\dot\Delta_2^3$ subject to the relations
\begin{subequations}\label{eq:JRelation}
\begin{gather}
1100-0011=1010-0101=1001-0110=0\label{sub:JRel1},\\
1100+0110+1010=0\label{sub:JRel2}.
\end{gather}
\end{subequations}
Relation~\eqref{sub:JRel1} states that opposite edges are equal, and \eqref{sub:JRel2} states that the sum of the $3$ edges meeting at a vertex is $0$.

We endow $J_{\Delta^3_2}$ with the skew symmetric bilinear form given by
\begin{equation}\label{eq:SymplecticForm}
\begin{aligned}
\langle1100,0110\rangle=\langle0110,1010\rangle=\langle1010,1100\rangle&=1\\
\langle0110,1100\rangle=\langle1010,0110\rangle=\langle1100,1010\rangle&=-1.
\end{aligned}
\end{equation}
Note that $\langle,\rangle$ is non-singular. Let
\begin{equation}
J_n(\T)=\bigoplus_{\Delta\in\T}\bigoplus_{s\in\Delta_{n-2}^3}J_{\Delta_2^3},
\end{equation}
be a direct sum of copies of $J_{\Delta^3_2}$, one for each subsimplex of each simplex of $\T$. Note that $J_n(\T)$ is generated by the set of all edges of all subsimplices of the simplices of $\T$. We represent a generator as a tuple $(\Delta,s,e)$. We extend the bilinear form $\langle,\rangle$ in the natural way, making the direct sum orthogonal. 
\begin{remark} When $n=2$, $J_n(\T)$ equals the space $J$ considered by Neumann~\cite[Section~4]{NeumannComb}.
\end{remark}
Let $L_n(\T)$ denote the free abelian group on the non-vertex integral points of $\T$.
Consider the map
\begin{equation}
\beta\colon L_n(\T)\to J_n(\T),\qquad p=\{(t,\Delta)\}\mapsto\sum_{(\Delta,t)\in p}\sum_{e+s=t}(\Delta,s,e).
\end{equation}
Using \eqref{sub:JRel1}, we can write $\beta(p)$ as
\begin{equation}
\sum_{\Delta\in\T,\,s\in\Delta_{n-2}^3}A_{p,s}'(\Delta,s,1100)+\sum_{\Delta\in\T,\,s\in\Delta_{n-2}^3}B_{p,s}'(\Delta,s,0110)+\sum_{\Delta\in\T,\,s\in\Delta_{n-2}^3}C_{p,s}'(\Delta,s,1010),
\end{equation}
where the entries of $A'$, $B'$ and $C'$ are all either $0$, $1$, or $2$. Using \eqref{sub:JRel2}, this further simplifies to
\begin{equation}\label{eq:Beta}
\beta(p)=\sum_{\Delta\in\T,\,s\in\Delta_{n-2}^3}A_{p,s}(\Delta,s,1100)+\sum_{\Delta\in\T,\,s\in\Delta_{n-2}^3}B_{p,s}(\Delta,s,0110).
\end{equation}
Note that the matrices $A'$, $B'$, $C'$, $A$ and $B$ are exactly those given by \eqref{eq:Unsimplified} and \eqref{eq:Simplified}. 

We identify $L_n(\T)$ with its dual via the natural basis, and $J_n(\T)$ with its dual via $\langle,\rangle$.
\begin{lemma}
The dual $\beta_n^*\colon J_n(\T)\to L_n(T)$ of $\beta_n$ is given by
\begin{equation}\label{eq:BetaStar}
\begin{aligned}
(\Delta,s,1100)&\mapsto [(\Delta,s+1001)]+[(\Delta,s+0110)]-[(\Delta,s+1010)]-[(\Delta,s+0101)]\\
(\Delta,s,0110)&\mapsto [(\Delta,s+1010)]+[(\Delta,s+0101)]-[(\Delta,s+1100)]-[(\Delta,s+0011)]\\
(\Delta,s,1010)&\mapsto [(\Delta,s+1100)]+[(\Delta,s+0011)]-[(\Delta,s+1001)]-[(\Delta,s+0110)],
\end{aligned}
\end{equation}
where $[(\Delta,t)]$ denotes the integral point determined by $(\Delta,t)$.
\end{lemma}
\begin{proof}
This is an immediate consequence of \eqref{eq:Beta} and \eqref{eq:SymplecticForm}.
\end{proof}

One can view the map geometrically as in Figure~\ref{fig:Signs}. The orientation of $\Delta$ determines which signs are positive.

\begin{figure}[htb]
\begin{center}
\begin{minipage}[b]{0.45\textwidth}
\begin{center}
\scalebox{0.75}{
\input{figures_gen/Signs.tex}}
\end{center}
\end{minipage}
\hfill
\begin{minipage}[b]{0.45\textwidth}
\begin{center}
\scalebox{0.75}{
\input{figures_gen/FacePointSigns.tex}}
\end{center}
\end{minipage}
\\
\begin{minipage}[t]{0.45\textwidth}
\caption{The signs of the terms in $\beta^*(\Delta,s,e)$.}\label{fig:Signs}
\end{minipage}
\begin{minipage}[t]{0.45\textwidth}
\caption{The signs of the terms in $\beta^*\circ\beta([(\Delta_0,t_0])$ coming from $\Delta_0$.}\label{fig:FacePointSigns}
\end{minipage}
\end{center}
\end{figure}


The elements $(\Delta,s,1100)$ and $(\Delta,s,0110)$ provide a basis for $J_n(\T)$. We fix an ordering such that $(\Delta,s,1100)>(\Delta',s',0110)$. In this basis, the form $\langle,\rangle$ becomes the standard symplectic form on $\Z^{2r}$ given by $\left(\begin{smallmatrix}0&I\\-I&0\end{smallmatrix}\right)$.

\begin{proposition}\label{prop:BetaStarBeta}
We have a chain complex
\begin{equation}
\xymatrix{L_n(\T)\ar[r]^\beta& J_n(\T)\ar[r]^{\beta^*}&L_n(\T),}
\end{equation}
i.e.~the map $\beta^*\circ\beta=0$.
The matrix representation of $\beta$ is the transpose of $(A\vert B)$, and the matrix representation of $\beta^*$ is the transpose of the coefficient matrix of the monomial map $\mu$ relating the Ptolemy coordinates and the shapes.
\end{proposition}
\begin{proof}
The proof is similar (in some sense dual) to the proof of Theorem~\ref{thm:PtolemyYieldsShapes}. 
Let $p=\{(\Delta_0,t_0),\dots,(\Delta_k,t_k)\}$ be an edge point. Let $s_k$ be the unique subsimplex of $\Delta_k$ having $t_k$ as an edge point. The triangulation induces a gluing of the simplices $s_k$ along a common edge as in Figure~\ref{fig:EdgeLocal}. Viewed from the top, this configuration looks like Figure~\ref{fig:EdgePointCancellation}. The signs indicated are the signs of the integral points involved in $\beta_n^*\circ\beta(p)$.
It follows that all signs cancel out.

\begin{figure}[htb]
\begin{center}
\scalebox{0.75}{
\input{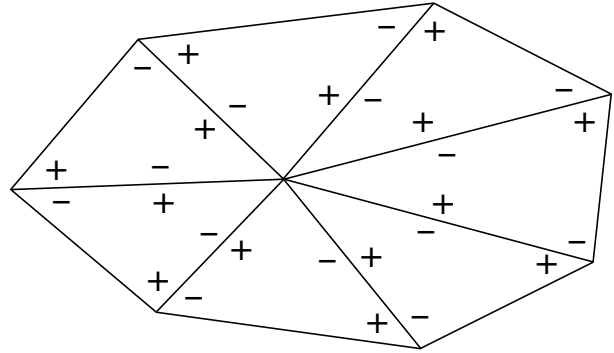}}
\end{center}
\caption{Cancellation of terms in $\beta^*\circ\beta(p)$ for an edge point $p$.}\label{fig:EdgePointCancellation}
\end{figure}

Let $p=\{(\Delta_0,t_0),(\Delta_1,t_1)\}$ be a face point. For simplicity we assume that $t_0=\alpha+1110$ is on the face opposite vertex $3$ of $\Delta_0$ (the other cases are similar). Then $t_0$ is an edge point of exactly $3$ subsimplices, $s_1=\alpha+0010$, $s_2=\alpha+1000$ and $s_3=\alpha+0100$. The terms of $\beta_n(p)$ coming from $(\Delta_0,t_0)$ are then
\begin{equation}
(\Delta_0,s_1,1100)+(\Delta_0,s_2,0110)+(\Delta_0,s_3,1010).
\end{equation}
Applying $\beta^*$ we obtain
\begin{equation}
\begin{aligned}
&[(\Delta_0,s_1+1001)]+[(\Delta_0,s_1+0110)]-[(\Delta_0,s_1+1010)]-[(\Delta_0,s_1+0101)]\\
+&[(\Delta_0,s_2+1010)]+[(\Delta_0,s_2+0101)]-[(\Delta_0,s_2+1100)]-[(\Delta_0,s_2+0011)]\\
+&[(\Delta_0,s_3+1100)]+[(\Delta_0,s_3+0011)]-[(\Delta_0,s_3+1001)]-[(\Delta_0,s_3+0110)],
\end{aligned}
\end{equation}
which equals
\begin{multline}
[(\Delta_0,\alpha+2010)]+[(\Delta_0,\alpha+1200)]+[(\Delta_0,\alpha+0120)]\\-[(\Delta_0,\alpha+2100)]-[(\Delta_0,\alpha+0210)]-[(\Delta_0,\alpha+1020)].
\end{multline}
Note that all terms are integral points lying on the same face as $t_0$. The signs are indicated in Figure~\ref{fig:FacePointSigns}. Since $\T$ is oriented, the terms arising from $(\Delta_1,t_1)$ are the same, but appear with opposite signs. Hence, they cancel out.

Let $p=\{(\Delta,t)\}$ be an interior point, where $t=\alpha+1111$. We have
\begin{multline}
\beta(p)=(\Delta,\alpha+1100,0011)+(\Delta,\alpha+1100,0011)+(\Delta,\alpha+1100,0011)+\\(\Delta,\alpha+1100,0011)+(\Delta,\alpha+1100,0011)+(\Delta,\alpha+1100,0011).
\end{multline}
As in the proof of Lemma~\ref{lemma:InternalEquation}, the positive and negative terms of $\beta^*\circ\beta(p)$ both consist of all terms $[(\Delta,\alpha+\beta)]$, where $\beta$ is a permutation of $(0,1,1,2)$. Hence, all terms cancel out.

The last statement follows from~\eqref{eq:Beta}, and by comparing \eqref{eq:BetaStar} and \eqref{eqn:PtolemyToShapes}.
\end{proof}
\begin{corollary}
The rows of $(A\vert B)$ Poisson commute.\qed
\end{corollary}

\begin{lemma}\label{lemma:Rank}
If all the boundary components of $M$ are tori, the number of non-vertex integral points of $\T$ equals $t\binom{n+1}{3}$.
\end{lemma}
\begin{proof}
Let $e$, $f$, and $t$ denote the number of edges, faces and simplices of $\T$. Since all boundary components are tori, a simple Euler characteristic argument shows that $e=\frac{1}{2}f=t$. 
Using this, we have
\begin{equation}\label{eq:EulerEquation}
\vert \dot\T_n(\Z)\vert=(n-1)e+\frac{(n-1)(n-2)}{2}f+\frac{(n-1)(n-2)(n-3)}{6}t=t\binom{n+1}{3}
\end{equation}
as desired.
\end{proof}
\begin{corollary}
If all the boundary components of $M$ are tori, the matrix $(A\vert B)$ is $r\times 2r$, where $r=t\binom{n+1}{3}$ and $t$ is the number of simplices of $\T$.
\end{corollary}
\begin{proof}
By Lemma~\ref{lemma:Rank} the number of rows equals $r$. The number of columns equals $2t\vert\Delta_{n-2}^3(\Z)\vert$, which by \eqref{eq:Counting} equals $2t\binom{n+1}{3}=2r$. 
\end{proof}
This concludes the proof of Theorem~\ref{thm:soft}.

\subsection{Applications in quantum topology}
\label{sub.applicationsQT}

Recently, ideal triangulations $\CT$ of 3-manifolds $M$ with torus boundary 
components
and their gluing equations have found several applications
in quantum topology, and this has been a main motivation for our work. We
will list three applications here, and refer to the literature for more
details:

\begin{itemize}
\item[(a)] The {\em Quantum Riemann surfaces} of \cite{Di1}
\item[(b)] The {\em loop invariants} of \cite{DG}
\item[(c)] The {\em 3D index} of \cite{DGG1,DGG2,Ga-3Dindex}
\end{itemize}

The input of a quantum Riemann surface of \cite{Di1} is an ideal triangulation 
$\CT$ of a 3-manifold with torus boundary components, and the output is
a polynomial in $q$-commuting variables (one per meridian and longitude
of each torus boundary component). The operators generate a $q$-holonomic
ideal which depends on $\CT$ and ought to map to the gluing
equation variety $V_2(\CT)$ when $q=1$. 

The input of the loop invariants of \cite{DG} is a Neumann-Zagier datum
which consists of an ideal triangulation together with a solution of
the gluing equations whose image under the map \eqref{eq.dev2} is
the discrete faithful representation of $M$. The output is
a formal power series in a variable $\hbar$ with
coefficients rational functions on the image of the map  
\eqref{eq.dev2}. The coefficient of $\hbar$ in the above series ought to agree
with the non-abelian torsion of \cite{Porti} and the evaluation of the
series at the discrete faithful representation when $\hbar=2\pi i/N$, ought
to agree to all orders in $1/N$ with the asymptotics of the Kashaev invariant
\cite{Kashaev} of a hyperbolic knot complement.
 
The input of the 3D index of \cite{DGG1,DGG2} (see also the survey
article \cite{Ga-3Dindex}) is an ideal triangulation $\CT$ which supports
a strict angle structure. The output is a $q$-holonomic function $
I_{\CT}: \BZ^{2r} \longto \BZ((q^{1/2}))$ where $r$ is the number of torus
boundary components of $M$ and $\BZ((q^{1/2}))$ is the ring of Laurent 
series in $q$ with integer coefficients.

Using our Theorems \ref{thm:soft} and \ref{thm:ShortMainThm}, one can extend 
the above invariants to the case of representations in $\PGL(n,\BC)$.
For example, fix an ideal triangulation $\CT$ with $t$ tetrahedra
of a 3-manifold $M$ with torus boundary. Following \cite[Sec.1.2]{DG},
choose a pair of opposite edges of each subsimplex, remove $n-1$ gluing
equations which are dependent on the others, and replace them with the
$n-1$ cusp equations of the meridian to obtain matrices $r \times r$
$A_n$ and $B_n$ where $r=t\binom{n+1}{3}$. Let $z$ denote the $r$ vector
of shapes of solutions to the gluing equations. Following 
\cite[Defn.1.1]{DG} consider the Neumann-Zagier datum 
$\b_{\CT,n}=((A_n|B_n),z)$ and its enhanced version 
$\hat\b_{\CT,n}=((A_n|B_n),z,f)$ where $f$ is a choice of flattening of each
subsimplex. Following \cite[Defn.1.2]{DG} we define the 1-loop invariant
of $\hat\b_{\CT,n}$ by
\be 
\qquad\tau_{\CT,n} = 
\pm \frac12\det\big( A_n\Delta_{z''}
+B_n\Delta_z^{-1}\big)z^{f''}z''^{-f}
\ee
where $\Delta_z:=\diag(z_1,...,z_r)$ and 
$\Delta_{z''}:=\diag(z_1'',...,z_r'')$ are diagonal matrices, and  
$z^{f''}z''^{-f} := \prod_i z_i{}^{f''_i}z''_i{}^{-f_i}$. When $z$ is the
solution that comes from the discrete faithful representation of $M$,
then $\tau_{\CT,n}$ lies in the invariant trace field of $M$.
An exact computation is possibe using the SnapPy tools \cite{SnapPy}. 
We thank N. Dunfield for providing an automated code for exact computation. 
Let us give some examples.

\begin{example}
\label{ex:411loop}
The $4_1$ knot has invariant trace field $\BQ(x)$ where 
$$
x^2 - x + 1=0, \qquad x=\frac{1+i \sqrt{3}}{2}
$$
If $\tau_n=\tau_{4_1,n}$, the quotient $\tau_{n+2}/\tau_n$
appears to have lower complexity than $\tau_n$ and is given by
\begin{align*}
\tau_2 &= 1/2 - x &
\tau_3 &= 21/2
\\
\frac{\tau_4}{\tau_2} &= 6720 &
\frac{\tau_5}{\tau_3} &= 52147200
\\
\frac{\tau_6}{\tau_4} &= -5381278156800 &
\frac{\tau_7}{\tau_5} &= -7383730314510950400
\\
\frac{\tau_8}{\tau_6} &= -138731589652863387775795200 &
\end{align*}
\end{example}

\begin{example}
\label{ex:52.1loop}
Consider the knot $5_2$ with invariant trace field $\BQ(x)$ where
$$
x^3-x^2+1=0, \qquad x=0.877439 \dots -i \, 0.744862 \dots \,.
$$
The {\em mirror} $-(-2,3,7)$ of the $(-2,3,7)$ pretzel knot has the same volume
and the same invariant trace field as $5_2$.  
If $\tau_n=\tau_{5_2,n}$ and $\tau'_n=\tau_{-(-2,3,7),n}$ 
then we have
{\small
\begin{align*}
\tau_2 &= 1-3/2 x &
\tau'_2 &= -2 + 5  x -3  x^2
\\
\frac{\tau_3}{\tau_2} &= -63+ 57 x^2 &
\frac{\tau'_3}{\tau'_2} &= -120 -48  x + 144  x^2
\\
\frac{\tau_4}{\tau_3} &= 1536 -3072 x -9792 x^2 &
\frac{\tau'_4}{\tau'_3} &= 21984 + 5712  x -31008  x^2
\\
\frac{\tau_5}{\tau_4} &= -12831360 -10393200 x +  8023440 x^2 &
\frac{\tau'_5}{\tau'_4} &= -3196800 + 14346000  x +  24614400  x^2
\\
\frac{\tau_6}{\tau_5} &= -95788707840 -84869406720 x  &
\frac{\tau'_6}{\tau'_5} &= 158834390400 + 213765955200 x 
\\
 & \quad + 55161630720 x^2 & & \quad + 4444675200 x^2
\end{align*}
}
\end{example}

\section{Decorations}\label{sec:Decorations}
We refer to Garoufalidis-Thurston-Zickert~\cite{GaroufalidisThurstonZickert} or Zickert~\cite{ZickertDuke} for more details on decorations.
Let $G$ be a group and $H$ a subgroup of $G$.
\begin{definition}\label{def:Decoration}
Let $\Delta$ be an ordered $k$-simplex. A $G/H$-\emph{decoration} of $\Delta$ is an assignment of a left $H$-coset to each vertex of $\Delta$. We only consider decorations up to $G$-action, i.e.~we consider two decorations to be equal if they differ by left multiplication by an element in $G$.
We represent a decoration by a tuple $(g_0H,\dots,g_kH)$.
\end{definition}
If $G$ and $H$ are clear from the context, we refer to a $G/H$-\emph{decoration} as a decoration.
\begin{definition} A decoration of a triangulated manifold $(M,\mathcal T)$ is a decoration of each simplex of $\mathcal T$ such that if two faces with decorations represented by $(g_0H,g_1H,g_2H)$ and $(g_0^\prime H,g_1^\prime H,g_2^\prime H)$ are identified, the decorations must differ by left multiplication by a \emph{unique} element in $G$.
\end{definition}

\begin{remark} Since the fundamental group is generated by face pairings, a decoration determines a representation $\pi_1(M)\to G$ taking peripheral curves to conjugates of $H$. Moreover, every such representation can be decorated.
\end{remark}
\begin{remark}
One can define, more intrinsically, a decoration as an equivariant assignment of cosets to the vertices of the space obtained from the universal cover of $M$ by collapsing each boundary component to a point. 
\end{remark}

\begin{remark}\label{rm:DecorationsAndBundles} A representation determines a flat bundle $E$ over $M$. One can show (\cite[Prop.~4.6]{GaroufalidisThurstonZickert}) that a decoration corresponds to a reduction of the restriction of $E$ to $\partial M$ to a flat $H$ bundle. Two decorations determine the same reduction if and only if they are \emph{equivalent} in the sense of \cite[Def.~4.4]{GaroufalidisThurstonZickert}. We shall not need this here.
\end{remark}

\subsection{Generic decorations, Ptolemy coordinates and shapes}
For an element $g\in\GL(n,\C)$, let $\{g\}_i$ denote the ordered set consisting of the first $i$ column vectors of $g$.
\begin{definition}\label{def:generic}
A $\GL(n,\C)/N$-decoration $(g_0N,g_1N,g_2N,g_3N)$ on $\Delta^3_n$ is \emph{generic} if for each $(t_0,t_1,t_2,t_3)\in\Delta_n^3(\Z)$
\begin{equation}\label{eqn:PtolemyDets}
\det\big(\{g_0\}_{t_0}\cup\{g_1\}_{t_1}\cup\{g_2\}_{t_2}\cup\{g_3\}_{t_3}\big)\neq 0.
\end{equation}
Genericity of $\PGL(n,\C)/B$-decorations is defined similarly.
\end{definition}
The definition is obviously independent of the choice of coset representatives, and of the ordering of the tuple.

\begin{remark}\label{rm:FineTriangulation}
Although all representations can be decorated, some representations may only have non-generic decorations. However, after a single barycentric subdivision, every representation has a generic decoration. For $\SL(n,\C)/N$-decorations this is proved in \cite[Prop.~5.4]{GaroufalidisThurstonZickert}, and the proof for $\PGL(n,\C)/B$-decorations is similar. This proves the last statement of Theorem~\ref{thm:ShortMainThm}.
\end{remark}

\begin{lemma}[{Fock-Goncharov~\cite[Lemma~10.3]{FockGoncharov}; see also~\cite{GaroufalidisThurstonZickert}}]\label{lemma:DecorationToPtolemy}
A generic $\GL(n,\C)/N$-decoration $(g_0N,g_1N,g_2N,g_3N)$ of $\Delta^3_n$ induces a Ptolemy assignment
\begin{equation}\label{eqn:DecoToPtolemy}
c\colon\dot\Delta_n^3(\Z)\to \C^*,\quad t\mapsto \det\left(\{g_0\}_{t_0}\cup\{g_1\}_{t_1}\cup\{g_2\}_{t_2}\cup\{g_3\}_{t_3}\right).
\end{equation}\qed
\end{lemma}
\begin{corollary} We have a map
\begin{equation}
\mathcal C\colon\text{\{Generic $\GL(n,\C)/N$-decorations on $\Delta^3_n$\}}\to \text{\{Ptolemy assignments on $\Delta^3_n$\}}. 
\end{equation}\qed
\end{corollary}
Note that $\mathcal C$ is invariant under the left action by $\SL(N,\C)$.

\begin{lemma} \label{lemma:DecorationsAndPullbacks}
The map $\mathcal C$ is compatible with pullbacks, i.e., 
\begin{equation}\label{eq:DecorationsAndPullbacks}
\sigma^*(\mathcal C(g_0N,g_1N,g_2N,g_3N)=\mathcal C(g_{\sigma(0)}N,g_{\sigma(1)}N,g_{\sigma(2)}N,g_{\sigma(3)}N).
\end{equation}
\end{lemma}
\begin{proof}
Let $c=\mathcal C(g_0N,g_1N,g_2N,g_3N)$ and $c'=\mathcal C(g_{\sigma(0)}N,g_{\sigma(1)}N,g_{\sigma(2)}N,g_{\sigma(3)}N)$. Then
\begin{equation}
c_{\sigma(t)}=\det\left(\bigcup_{i=0}^3\{g_i\}_{t_{\sigma^{-1}(i)}}\right),\qquad c'_t=\det\left(\bigcup_{i=0}^3\{g_{\sigma(i)}\}_{t_i}\right).
\end{equation}
One easily checks that
\begin{equation}
\left(\bigcup_{i=0}^3\{g_i\}_{t_{\sigma^{-1}(i)}}\right)I_{\sigma,\sigma(t)}=\bigcup_{i=0}^3\{g_{\sigma(i)}\}_{t_i},
\end{equation}
from which it follows that $(\sigma^*c)_t=c_{\sigma(t)}\det(I_{\sigma,\sigma(t)})=c'_t$. This proves the result.
\end{proof}
\begin{corollary}
A generic $\SL(n,\C)/N$-decoration on $(M,\mathcal{T})$ induces a Ptolemy assignment on $(M,\mathcal{T})$.
\end{corollary}

\begin{proof}
We only need to show that Ptolemy coordinates are identified via the pullback. This follows from Lemma~\ref{lemma:DecorationsAndPullbacks}.
\end{proof}

\begin{remark}
A $\GL(N,\C)/N$-decoration does \emph{not} induce a Ptolemy assignment on $(M,\T)$.
\end{remark}

For $g\in\GL(N,\C)$, let $\bar{g}$ denote the image of $g$ in $\PGL(N,\C)$. 

\begin{proposition}\label{prop:DefofZ}
We have a well defined map
\begin{equation}
\begin{gathered}
\mathcal Z\colon\text{\{$\PGL(n,\C)/B$-decorations on $\Delta^3_n$\}}\to\text{\{shape assignments on $\Delta^3_n$\}},\\
(\bar{g}_0B,\bar{g}_1B,\bar{g}_2B,\bar{g}_3B)\mapsto\mu\circ\mathcal C(g_0N,g_1N,g_2N,g_3N),
\end{gathered}
\end{equation}
which is invariant under the left $\PGL(n,\C)$-action and compatible with pullbacks. 
\end{proposition}
\begin{proof}
Let $c=\mathcal C(g_0N,g_1N,g_2N,g_3N)$.
To prove that $\mathcal Z$ is well defined, we must prove that 
\begin{equation}\mu(c)=\mu(c'),\text{ with } c'=\mathcal C(g_0d_0N,\dots,g_3d_3N),
\end{equation}
where $d_i=d(\lambda^i_1,\dots,\lambda^i_n)$ are diagonal matrices. We prove that $\mu(c)_s^{1100}=\mu(c')^{1100}_s$, the other cases being similar.
Using \eqref{eqn:PtolemyDets} we have
\begin{equation}\mu(c')_s^{1100}=\frac{c'_{s+1001}c'_{s+0110}}{c'_{s+1010}c'_{s+0101}},\quad c'_t=\det(\bigcup\{g_id_i\}_{t_i})=\left(\prod_{j=0}^{3}\prod_{i=1}^{t_i} \lambda^j_i\right) c_t.
\end{equation}
Expanding each term, we obtain
\begin{equation}
\begin{gathered}
c'_{s+1001} = \left(\prod_{j=0}^{3}\prod_{i=1}^{s_i} \lambda^j_i\right) \lambda^0_{s_0+1} \lambda^3_{s_3+1} c_{s+1001},\quad
c'_{s+0110} = \left(\prod_{j=0}^{3}\prod_{i=1}^{s_i} \lambda^j_i\right) \lambda^1_{s_1+1} \lambda^2_{s_2+1} c_{s+0110}\\
c'_{s+1010} = \left(\prod_{j=0}^{3}\prod_{i=1}^{s_i} \lambda^j_i\right) \lambda^0_{s_0+1} \lambda^2_{s_2+1} c_{s+1010},\quad
c'_{s+0101} = \left(\prod_{j=0}^{3}\prod_{i=1}^{s_i} \lambda^j_i\right) \lambda^1_{s_1+1} \lambda^3_{s_3+1} c_{s+0101}.
\end{gathered}
\end{equation}
It now easily follows that 
\begin{equation}
\mu(c')_s^{1100}=\frac{c'_{s+1001}c'_{s+0110}}{c'_{s+1010}c'_{s+0101}}=\frac{c_{s+1001}c_{s+0110}}{c_{s+1010}c_{s+0101}}=\mu(c)_s^{1100}
\end{equation}
as desired. Invariance under left multiplication follows from the fact that $\det(\cup\{gg_i\}_{t_i})=\det(g)\det(\cup\{g_i\}_{t_i})$, and compatibility with pullbacks follows from the fact that both $\mu$ and $\mathcal C$ enjoy this property. 
\end{proof}

\begin{corollary}
A generic $\PGL(n,\C)/B$-decoration on $(M,\T)$ induces a shape assignment on $(M,\T)$.
\end{corollary}

\begin{proof}
By Proposition~\ref{prop:DefofZ}, we have a shape assignment on each simplex, and we must prove that these satisfy the generalized gluing equations. We proceed as in the proof of Theorem~\ref{thm:PtolemyYieldsShapes}. Let $K$ be the local model of an edge point as defined in Section~\ref{subsection:ProofForEdgeEqn}. We can pullback the decoration on $\T$ to a decoration on $K$ using the simplical map $\pi$. Since $K$ is simply connected, we can change the decoration of each simplex by left multiplication by an element in $\PGL(n,\C)$ such that vertices of simplices that get identified in $K$ carry the same coset. This does not affect the shapes. For each vertex in $K$ decorated by $gB$ we pick a lift $\tilde{g}N$ and apply $\mathcal C$ to get a Ptolemy assignment on $K$. By Lemma~\ref{lemma:edgelocalmodel} the shapes satisfy the edge equations. The result for the face, and interior gluing equations is similar.
\end{proof}

\section{The natural cocycle of a generic decoration}\label{sec:NaturalCocycles}
In this section, we introduce natural cocycles on $M$ arising from decorations. To define these we need two types of polyhedral decompositions of $M$, one by \emph{truncated simplices} and one by \emph{doubly truncated simplices}. We show that an $\SL(n,\C)/N$-decoration induces a natural cocycle on the truncated decomposition of $M$, and that a $\PGL(n,\C)/B$-decoration induces a natural cocycle on the doubly truncated decomposition of $M$. Later, we give explicit formulas in terms of the Ptolemy coordinates, respectively, shape coordinates.

\subsection{Truncated and doubly truncated simplices}
\begin{definition}
A \emph{truncated simplex} is a polyhedron obtained from a simplex by truncating its vertices. A \emph{doubly truncated simplex} is a polyhedron obtained from a simplex by first truncating the vertices and then truncating the edges.
\end{definition}
We refer to the edges of a truncated simplex as \emph{long} and \emph{short} edges, and the edges of a doubly truncated simplex as \emph{long}, \emph{middle}, and \emph{short} edges.

Note that a triangulation of $M$ induces a decomposition of $M$ into truncated simplices, as well as a decomposition of $M$ into doubly truncated simplices and \emph{prisms}. 
An \emph{ordering} of a simplex induces an orientation of the edges of the corresponding truncated simplex. Similarly, an \emph{orientation} of a simplex induces orientations of the edges of the corresponding doubly truncated simplex. Note that an ordering is required to obtain natural edge orientations on a truncated simplex.

\begin{figure}[htb]
\begin{center}
\begin{minipage}[b]{0.34\textwidth}
\scalebox{0.9}{
\input{figures_gen/TruncatedCocycle.tex}}
\end{minipage}
\hfill
\begin{minipage}[b]{0.34\textwidth}
\scalebox{0.9}{
\input{figures_gen/doublyTruncatedCocycle.tex}}
\end{minipage}
\hfill
\begin{minipage}[b]{0.3\textwidth}
\qquad
\scalebox{0.9}{
\input{figures_gen/doublyTruncatedCocycle_Prism.tex}}
\end{minipage}\\ \vspace{-0.4cm}
\begin{minipage}[t]{0.34\textwidth}
\caption{A truncated simplex. Edge orientations induced by the vertex ordering.}\label{fig:Truncated}
\end{minipage}
\begin{minipage}[t]{0.34\textwidth}
\caption{A doubly truncated simplex. Edge orientations induced by the orientation.}\label{fig:DoublyTruncated} 
\end{minipage}
\begin{minipage}[t]{0.3\textwidth}
\caption{A prism.}\label{fig:Prism}
\end{minipage}
\hfill
\end{center}
\end{figure}


\begin{remark}
We can view a doubly truncated simplex as the permutohedron of $S_4$. We can embed it in a standard simplex $\Delta^3_1$ as the convex hull of $\frac{1}{11}\Delta^3_{11}(\{0,1,3,7\})$, the set $\frac{1}{11}\Delta^3_{11}(\{0,1,3,7\})$ being the vertex set. With this embedding, the long edges are twice as long as the middle edges, which are again twice as long as the short edges (which have length $\frac{\sqrt{2}}{11}$). Similarly, we may view a truncated simplex as the convex hull of $\frac{1}{5}\Delta^3_{5}(\{0,1,4\})$.
\end{remark}




\subsection{Cocycles} Let $G$ be a group and let $X$ be any space with a polyhedral decomposition.

\begin{definition}\label{def:GCocycle}
A \emph{$G$-cocycle} on $X$ is an assignment of elements in $G$ to the oriented edges of $X$ such that the product around each face is $1\in G$ and such that reversing the orientation of an edge replaces the labeling by its inverse.
\end{definition}

\begin{definition}
Let $\tau$ be a $G$-valued $0$-cochain on $X$, i.e.~a function from the vertices to $G$. The \emph{coboundary} of $\tau$ is the $G$-cocycle
\begin{equation}
\delta\tau(\langle v_0,v_1\rangle)=(\tau^{v_0})^{-1}\tau^{v_1},
\end{equation}
where $\langle v_0,v_1\rangle$ is the edge from $v_0$ to $v_1$ and $\tau^{v_i}$ is the value of $\tau$ at $v_i$.
The \emph{coboundary action} of $\tau$ on $G$-cocycles is given by taking $\sigma$ to the cocycle
\begin{equation}\label{eq:CoboundaryAction}
\tau\sigma\colon\langle v_0,v_1\rangle\mapsto (\tau^{v_0})^{-1}\sigma(\langle v_0,v_1\rangle)\tau^{v_1}.
\end{equation}
\end{definition}

Given a simplex $\Delta$, let $\truncsimp$ and $\dtruncsimp$ denote the corresponding truncated, and doubly truncated simplices.
\begin{definition} Let $H\subset G$ be groups. A $(G,H)$-cocycle on $\truncsimp$ is a $G$-cocycle where short edges are labed by elements in $H$. 
\end{definition}

\begin{definition} Let $K\subset H\subset G$ be groups. A $(G,H,K)$-cocycle on $\dtruncsimp$ is a $G$-cocycle where short edges are labeled by elements in $K$, and middle edges by elements in $H$.
\end{definition}

\begin{remark}\label{rm:InducedCocycle} Note that every $(G,H)$-cocycle on $\truncsimp$ can be obtained from a unique $(G,H,\{e\})$-cocycle on $\dtruncsimp$ by collapsing the short edges. We shall thus always regard a cocycle on $\truncsimp$ as a cocycle on $\dtruncsimp$.
\end{remark}

\subsubsection{Labeling conventions}\label{sec:LabelingConv}
We index the vertices of $\truncsimp$ by ordered pairs of distinct vertices of $\Delta$, $v_0v_1$ being the vertex near $v_0$ on the edge to $v_1$. 
We index the vertices of $\dtruncsimp$ by ordered triples of distinct vertices of $\Delta$, $v_0v_1v_2$ being the vertex, whose closest vertex in $\Delta$ is $v_0$, closest edge $v_0v_1$, and closest face $v_0v_1v_2$. Given a cocycle on $\dtruncsimp$, we use $\alpha$'s to denote the labeling of long edges, $\beta$'s for the middle edges, and $\gamma$'s for the short edges. Note that an edge of each type is uniquely determined by its initial vertex. This gives a unique labeling scheme, e.g.~the long edge from $v_0v_1v_2$ to $v_1v_0v_2$ is labeled by $\beta^{v_0v_1v_2}$. Similarly, if $\tau$ is a $0$-cochain, the value at $v_0v_1v_2$ is denoted by $\tau^{v_0v_1v_2}$.
We shall not need a labeling scheme for cocycles on truncated simplices. By Remark~\ref{rm:InducedCocycle} we can regard these as cocycles on the corresponding doubly truncated simplices. 


%

\subsection{The natural cocycle of a generic decoration}
We now show that sufficiently generic decorations naturally give rise to cocycles on $M$. 

\begin{definition} A pair $(g_0N,g_1N)$ of $N$-cosets in $\SL(n,\C)$ is \emph{sufficiently generic} if there exists  a (necessarily unique) $g\in\SL(n,\C)$ such that
\begin{equation}
(g_0N,g_1N)=g(N,qN), \text{ with $q$ counter-diagonal}.
\end{equation}
A tuple is sufficiently generic if it is pairwise sufficiently generic.
\end{definition}
Let $N^-$ denote the lower triangular matrices in $\SL(n,\C)$ with $1$'s on the diagonal.
\begin{definition}
An element $n\in B$ is \emph{normalized} if the \emph{last} column vector consists of $1$'s. An element $n_-\in N^-$ in \emph{normalized} if the \emph{first} column vector consists of $1$'s.
\end{definition}
\begin{definition}
A triple $(g_0B,g_1B,g_2B)$ of $B$-cosets in $\PGL(n,\C)$ is \emph{sufficiently generic} if there exists a (necessarily unique) $g\in\PGL(n,\C)$ such that
\begin{equation}
(g_0B,g_1B,g_2B)=g(B,q_1B,n_-B), \text{ with $n_-\in N^-$ normalized}.
\end{equation}
A tuple is sufficiently generic if each triple is sufficiently generic.
\end{definition}

\begin{remark} A simple exercise in linear algebra shows that generic (as in Definition~\ref{def:generic}) implies sufficiently generic. 
\end{remark}

\begin{definition}\label{def:NaturalSLCocycle}
Let $(g_0N,g_1N,g_2N,g_3N)$ be a generic $\SL(n,\C)/N$-decoration on $\Delta^3_n$. The \emph{natural $(\SL(n,\C),N)$-cocycle} on $\overline{\Delta^3_n}$ is the coboundary of the unique $0$-cochain $\tau$  satisfying
\begin{equation}
(g_{v_0}N,g_{v_1}N)=\tau^{v_0v_1}(N,qN) \text{ with $q$ counter diagonal}.
\end{equation}
\end{definition} 
This defines the map $\mathcal L_{\alpha\beta}$ in~\eqref{eqn:DiagramIntro}. It follows immediately from the definition that the natural cocycle labels short edges by elements in $N$ and long edges by counter diagonal elements. 
Let $q_1$ denote the counter diagonal matrix whose non-zero entries are all $1$. Note that $q_1=q_1^{-1}$.
\begin{definition}\label{def:NaturalPGLCocycle}
Let $(g_0B,g_1B,g_2B,g_3B)$ be a generic decoration on a simplex. The \emph{natural $(\PGL(n,\C),B,H)$-cocycle} on $\overline{\overline{\Delta^3_n}}$ is the coboundary of the unique $0$-cochain $\tau$ satisfying
\begin{equation}
(g_{v_0}B,g_{v_1}B,g_{v_2}B)=\tau^{v_0v_1v_2}(B,q_1B,n_-B) \text{ with $n_-\in N^- $ normalized}.
\end{equation}
\end{definition} 

This defines the map $\mathcal L_{\alpha\beta\gamma}$ in~\eqref{eqn:DiagramIntro}.

\begin{remark}\label{rm:CocycleToDecoration}
Note that a $(G,H)$-cocycle $\sigma$ on $\truncsimp$ determines a $G/H$-decoration $D$ on $\Delta$. We say that $D$ is \emph{compatible} with $\sigma$. To see this note that $\sigma$ is the coboundary of a $0$-cochain $\tau$ on $\truncsimp$, which is unique up to left multiplication by an element in $G$. The value of $\tau$ at the vertices near a vertex of $\Delta$ are all in the same $H$-coset. Hence, we have a decoration on $\Delta$. Similarly, a $(G,B,H)$-cocycle on $\dtruncsimp$ determines a $G/B$-decoration on $\Delta$. It follows that the maps $\mathcal L_{\alpha\beta}$ and $\mathcal L_{\alpha\beta\gamma}$ are bijective with explicit inverses.
\end{remark}

\begin{lemma}\label{lemma:UniversalProperty} Let $D=(g_0B,g_1B,g_2B,g_3B)$ be a generic decoration on a simplex. The \emph{natural $(\PGL(n,\C),B,H)$-cocycle} is the unique 
cocycle, which is compatible with the decoration and satisfies
\begin{enumerate}[(i)]
\item Short edges are labeled by elements in $H$.\label{short}
\item Middle edges are labeled by normalized elements in $B$.\label{middle}
\item Long edges are labeled by $q_1$.\label{long}
\end{enumerate}
\end{lemma}

\begin{proof}
We first show that the natural cocycle satisfies the three conditions.
It is enough to prove this for a single edge of each type. We may assume that $D=(B,q_1,B,n_-B,m_-B)$ where $n_-, m_-\in N^-$ and $n_-$ is normalized. Then $\tau^{012}=1$, so for each edge starting at $012$, we only need to compute the value of $\tau$ at the end point. Since $(q_1B,B,n_-B)=q_1(B,q_1B,q_1n_-B)$, and the first column vector of $q_1n_-$ consists of $1$'s, it follows that $\tau^{102}=q_1$, proving the result for the long edges.
Since the stabilizer of $B$ is $B$, it follows that $\tau^{021}\in B$, proving the result for the middle edges.
Finally, $\tau^{013}$ is the unique element in $H$ such that $hm_-h^{-1}$ is normalized, proving the result for the short edges.

Let $\sigma_1$ and $\sigma_2$ be two cocycles satisfying the required conditions. Since any two cocycles differ by the coboundary action, $\sigma_2=\eta\sigma_1$ for some coboundary $\eta$. Since long edges are labeled by $q_1$ and since the cocycles determine the same decoration, we may assume that $\eta$ takes values in $H$. It is now elementary to check that if $\eta$ is not the identity, either~\eqref{middle} or~\eqref{long} fails.
\end{proof}

\begin{remark}
Note that for a generic decoration on a triangulation of $M$, the natural cocycles on each simplex fit together to form a natural cocycle on $M$.
\end{remark}

\section{Explicit formulas for the natural cocycles}\label{sec:ExplicitFormulas}
We now show that the cocycle associated to a $\PGL(n,\C)/B$-decoration $D$ is determined by the shape assignment $\mathcal Z(D)$. The reader should keep in mind the diagram~\eqref{eqn:DiagramIntro}. 

Define
\begin{equation}
\begin{gathered}
q(a_1,\dots,a_n)=\begin{pmatrix}&&a_n\\&\iddots&\\a_1&&\end{pmatrix},\qquad d(a_1,\dots,a_n)=\begin{pmatrix}a_1&&\\&\ddots&\\&&a_1\end{pmatrix}\\
q_1=q(1,\dots,1),\qquad d_{\pm 1}=d\big(\{(-i)^{n-k}\}_{k=1}^n),\qquad H_i(x)=d\big(\overbrace{x,\dots,x}^i,1,\dots,1\big).
\end{gathered}
\end{equation}  
Letting $E_{i,i+1}$ be the matrix with a $1$ as the $(i,i+1)$ entry, and zeros elsewhere, we define
\begin{equation}
x_i(t)=I+tE_{i,i+1}.
\end{equation}

\subsection{Diamond and ratio coordinates}

It is shown in Garoufalidis-Thurston-Zickert~\cite{GaroufalidisThurstonZickert} that the short edges of the natural cocycle of a generic $\SL(n,\C)/N$-decoration are given by \emph{diamond coordinates}, and that the long edges are given by ratios of two Ptolemy coordinates. We review these results below.



\begin{definition}
Let $c$ be a Ptolemy assignment on $\Delta_n^3$. For each vertex $v_0v_1v_2$ of $\overline{\overline{\Delta^3_n}}$ and each $\alpha\in\Delta^3_{n-2}(\Z)$ on the face containing $v_0v_1v_2$, we associate a \emph{diamond coordinate}
\begin{equation}\label{eq:DiamondDef}
d^{v_0v_1v_2}_\alpha(c)= - \varepsilon_<^{v_0v_1v_2} \frac{c_{\alpha+2v_0}c_{\alpha+v_1+v_2}}{c_{\alpha+v_0+v_1}c_{\alpha+v_0+v_2}}.
\end{equation}
Here $\varepsilon_<^{v_0v_1v_2}$ is the sign of the $S_3$ permutation required to bring the sequence $v_0,v_1,v_2$ into lexicographic order.
\end{definition}

\begin{definition}
Let $c$ be a Ptolemy assignment on $\Delta_n^3$. For each vertex $v_0v_1v_2$ of $\overline{\overline{\Delta^3_n}}$ and each point $kv_0+lv_1$ on the long edge containing $v_0v_1v_2$, we associate a \emph{ratio coordinate}
$$e^{v_0v_1}_{kv_0+lv_1}(c)= (-1)^{l} ~ \frac{c_{kv_0+(l+1)v_1}}{c_{(k+1)v_0+lv_1}}\quad\mbox{where}\quad k+l=n-1.$$
\end{definition}

\begin{figure}[htb]
\begin{center}
\begin{minipage}[c]{0.4\textwidth}
\input{figures_gen/diamondCoordinatesDef.tex}
\end{minipage}
\hfill
\begin{minipage}[c]{0.4\textwidth}
\input{figures_gen/FigRatioCoordinates.tex}
\end{minipage}
\hfill
\\ \vspace{-0.4cm}
\begin{minipage}[t]{0.4\textwidth}
\caption{Diamond coordinates.}\label{fig:DiamondCoordinates}
\end{minipage}
\begin{minipage}[t]{0.4\textwidth}
\caption{Ratio coordinates.}\label{fig:RatioCoordinates}
\end{minipage}
\end{center}
\end{figure}
\begin{notation}When $c$ is clear from the context, we suppress it from the notation, i.e.~we write $d^{v_0v_1v_2}_\alpha$ and $e^{v_0v_1}_{kv_0+lv_1}$ instead of $d^{v_0v_1v_2}_\alpha(c)$ and $e^{v_0v_1}_{kv_0+lv_1}(c)$.
\end{notation}

As explained in Remark~\ref{rm:InducedCocycle}, we can view the natural cocycle of an $\SL(n,\C)/N$-decoration as a cocycle on $\overline{\overline{\Delta^3_n}}$. We thus employ the labeling conventions of Section~\ref{sec:LabelingConv}.
\begin{proposition}[Garoufalidis-Thurston-Zickert~\cite{GaroufalidisThurstonZickert}]\label{prop:NaturalCocycleLongShort}
The natural cocycle $\mathcal L_{\alpha\beta}(D)$ of a generic $\SL(n,\C)/N$-decoration $D$ on $\Delta^3_n$ is given in terms of the Ptolemy assignment $\mathcal C(D)$ by
\begin{equation}\label{eq:betaTrunc}
\begin{gathered}
\gamma^{v_0v_1v_2}=\id,\qquad \beta^{v_0v_1v_2}=\prod_{(\alpha_0,\alpha_1,\alpha_2)\in\Delta^2_{n-2}(\Z)} x_{\alpha_1+1}\left(d^{v_0v_1v_2}_{\alpha_1v_0+\alpha_2v_1+\alpha_0v_2}\right)\\\alpha^{v_0v_1v_2}=q(e^{v_0v_1}_{(n-1)v_0},e^{v_0v_1}_{(n-2)v_0+v_1},\dots,e^{v_0v_1}_{(n-1)v_1}),\qquad
\end{gathered}
\end{equation}
In the product, the order of the factors is given by the lexicographic order on $\Delta^2_{n-2}(\Z)$.\qed
\end{proposition}

\begin{remark} It is convenient to introduce the notation
\begin{equation}
d_{k,i}=d^{v_0v_1v_2}_{(i-1)v_0+(n-i-k)v_1+(k-1)v_2}.
\end{equation} 
With this notation, the formula for the middle edge becomes
\begin{equation}
\beta^{v_0v_1v_2}=\prod_{k=1}^{n-1}\prod_{i=1}^{n-k}x_i(d_{k,i}).
\end{equation}
This agrees with the notation in~\cite{GaroufalidisThurstonZickert}. Although this notation is convenient, it does not behave properly under reordering.
\end{remark}

\begin{figure}[htb]
\begin{center}
\input{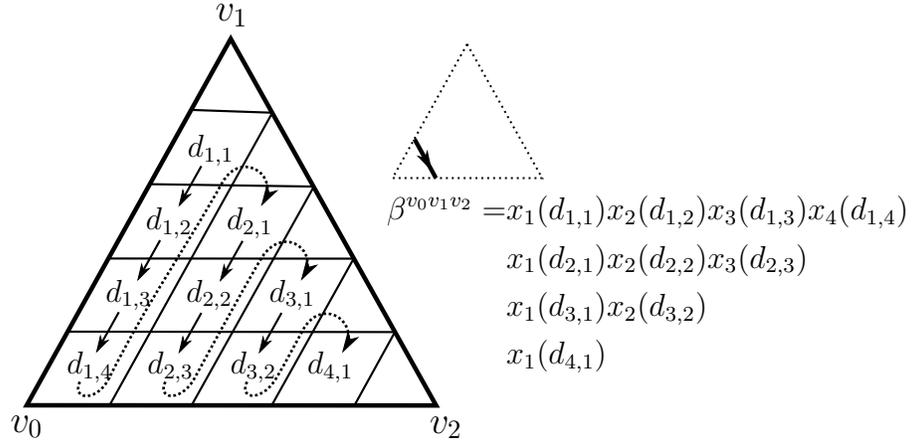}
\caption{Factorization of a middle edge $\beta^{v_0v_1v_2}$ in terms of diamond coordiates.}\label{fig:DiamondFactorization}
\end{center}
\end{figure}

\subsubsection{Behavior under reordering}
\begin{lemma} The diamond coordinates of $c$ and $\sigma^*(c)$ are related by
\begin{equation}
d(\sigma^*(c))_{\alpha}^{v_0v_1v_2}=d(c)_{\sigma(\alpha)}^{\sigma(v_0)\sigma(v_1)\sigma(v_2)}.
\end{equation}
\end{lemma}
\begin{proof}
It is enough to prove this for $v_i=i$ and $\sigma=\sigma_{01}$, $\sigma=\sigma_{12}$ and $\sigma=\sigma_{23}$. We prove it for $\sigma=\sigma_{01}$, the other cases being similar. Letting $t=\alpha+e_0+e_1$ and $\alpha'=\sigma_{01}(\alpha)$, and using that $\varepsilon_<^{012}=1$ and $\varepsilon_<^{102}=-1$, we have
\begin{equation}
\begin{aligned}
d(\sigma_{01}^*(c))^{012}_\alpha&=-\frac{\sigma_{01}^*(c)_{\alpha+2e_0}\sigma_{01}^*(c)_{\alpha+e_1+e_2}}{\sigma_{01}^*(c)_{\alpha+e_0+e_1}\sigma_{01}^*(c)_{\alpha+e_0+e_2}}\\
&=-\frac{\det(I_{\sigma_{01},\sigma_{01}(t+e_0-e_1)})\det(I_{\sigma_{01},\sigma_{01}(t+e_2-e_0)})}{\det(I_{\sigma_{01},\sigma_{01}(t)})\det(I_{\sigma_{01},\sigma_{01}(t+e_2-e_1)})}\frac{c_{\alpha'+2e_1}c_{\alpha'+e_0+e_2}}{c_{\alpha'+e_1+e_0}c_{\alpha'+e_1+e_2}}\\
&=-\frac{(-1)^{(t_0+1)(t_1-1)}(-1)^{(t_0-1)t_1}}{(-1)^{t_0t_1}(-1)^{t_0(t_1-1)}}\frac{c_{\alpha'+2e_1}c_{\alpha'+e_0+e_2}}{c_{\alpha'+e_1+e_0}c_{\alpha'+e_1+e_2}}\\
&=-(-1)\frac{c_{\alpha'+2e_1}c_{\alpha'+e_0+e_2}}{c_{\alpha'+e_1+e_0}c_{\alpha'+e_1+e_2}}\\
&=d(c)_{\alpha'}^{102}.
\end{aligned}
\end{equation}
This proves the result.
\end{proof}

\begin{lemma}
The ratio coordinates of $c$ and $\sigma^*c$ are related by
\begin{equation}
\sigma^*(e)^{v_0v_1}_{\alpha}= \varepsilon^{n-1} e^{\sigma(v_0)\sigma(v_1)}_{\sigma(\alpha)}
\end{equation}
where $\varepsilon$ is a sign depending on whether or not $\sigma$ flips the orientation of the edge $v_0v_1$. 
\end{lemma}
\begin{proof}
We shall not need this, so we leave the proof to the reader.
\end{proof}

\subsection{$X$-coordinates}
We define $X$-coordinates for Ptolemy assignments and shape assignments. These are defined for face points and agree with the $X$-coordinates of Fock and Goncharov~\cite[p.~133]{FockGoncharov}.

The natural $A_4$ action on vertices of $\dtruncsimp$ has two orbits. Let $\varepsilon_\circlearrowleft^{v_0v_1v_2}$ be a sign, which is positive if and only if $v_0v_1v_2$ is in the orbit of $012$.
\begin{definition} Let $c$ be a Ptolemy assignment and let $t\in\Delta_n^3(\Z)$ be a face point. 
The \emph{$X$-coordinate} at $t$ is given by
\begin{equation}
X_t=\prod_{t\in \mathrm{face}(v_0v_1v_2)}c^{\varepsilon_\circlearrowleft^{v_0v_1v_2}}_{t+v_0-v_1}
\end{equation} 
where the product is taken over the six ordered triples of vertices $v_0, v_1, v_2$ spanning the face containing $t$.
\end{definition}
As an example, the $X$-coodinate of $t=(t_0,t_1,t_2,0)$ is given by
\begin{equation}\label{eq:XExample}
X_{(t_0,t_1,t_2,0)}=\frac{c_{t+e_0-e_1}c_{t+e_1-e_2}c_{t+e_2-e_0}}{c_{t+e_0-e_2}c_{t+e_1-e_0}c_{t+e_2-e_1}}.
\end{equation}


\begin{definition} 
Let $z$ be a shape assignment on $\Delta^3_n$ and let $t$ be a face point spanned by $v_0$, $v_1$ and $v_2$. The \emph{$X$-coordinate} at $t$ is given by
\begin{equation}\label{eq:XandProducts}
X_t=-\prod_{s+e=t}z^e_s.
\end{equation}
\end{definition} 

\begin{remark}\label{rm:HalfofFaceEq}
Note that the product~\eqref{eq:XandProducts} consists of half of the terms involved in a face equation. More precisely, if $v_0v_1v_2\in\overline{\overline{\Delta_0}}$ is glued to $w_0w_1w_2\in\overline{\overline{\Delta_1}}$, the face equations are given by 
\begin{equation}
X_{t_0v_0+t_1v_1+t_2v_2}^{\varepsilon_\circlearrowleft^{v_0v_1v_2}}=X_{t_0w_0+t_1w_1+t_2w_2}^{\varepsilon_\circlearrowleft^{w_0w_1w_2}}.
\end{equation}
 Note that $\varepsilon_\circlearrowleft^{v_0v_1v_2}$ and $\varepsilon_\circlearrowleft^{w_0w_1w_2}$ are equal if and only if the face pairing preserves orientation. For oriented triangulations the signs are always opposite.
\end{remark}

\begin{lemma} The $X$-coordinates transform as the shapes under reordering, i.e.~we have
\begin{equation}
X(\sigma^*(c))_t=X(c)^{\sgn(\sigma)}_{\sigma(t)},\qquad X(\sigma^*(z))_t=X(z)^{\sgn(\sigma)}_{\sigma(t)}.
\end{equation}
\end{lemma}
\begin{proof}
Unwinding the definitions, we have
\begin{multline}
X(\sigma^*c)_t=
\prod_{t\in \mathrm{face}(v_0v_1v_2)}(\sigma^* c)^{\varepsilon_\circlearrowleft^{v_0v_1v_2}}_{t+v_0-v_1} =
\prod_{t\in\mathrm{face}(v_0v_1v_2)} \det(I_{\sigma,\sigma(t+v_0-v_1)}) c^{\varepsilon_\circlearrowleft^{v_0v_1v_2}}_{\sigma(t)+\sigma(v_0)-\sigma(v_1)}=\\
\left(\prod_{t\in \mathrm{face}(v_0v_1v_2)} \det(I_{\sigma,\sigma(t+v_0-v_1)})\right) X(c)_{\sigma(t)}^{\sgn(\sigma)} =
 X(c)_{\sigma(t)}^{\sgn(\sigma)}.
\end{multline}
The fact that the product of determinants equals $1$ follows from Remark~\ref{rm:ShuffleOdd}, which implies that $\det(I_{\sigma,\sigma(t+v_0-v_1)})=\det(I_{\sigma,\sigma(t+v_1-v_0)})$. Since both are $\pm 1$, their product is $1$. The second equation is obvious.
\end{proof}

\begin{lemma} The $X$-coordinates of a Ptolemy assignment $c$ agree with the $X$-coordinates of the corresponding shape assignment $\mu(c)$.
\end{lemma}
\begin{proof}
We must prove that
\begin{equation}
-\prod_{s+e=t}\mu(c)^e_s=\prod_{t\in \mathrm{face}(v_0v_1v_2)}c^{\varepsilon_\circlearrowleft^{v_0v_1v_2}}_{t+v_0-v_1}.
\end{equation}
By compatibility under reordering, it is enough to prove this for $t=(t_0,t_1,t_2,0)=t_0e_0+t_1e_1+t_2e_2$.
Let $t=\alpha+(1110)$. By \eqref{eq:Xcoordinate},
\begin{equation}
\begin{aligned}
-\prod_{s+e=t}\mu(c)^e_s&=z^{1100}_{\alpha+0010}z^{0110}_{\alpha+1000}z^{1010}_{\alpha+0010}\\
&=\frac{c_{\alpha+2010,0}c_{\alpha+1200,0}c_{\alpha+0120,0}}{c_{\alpha+2100,0} c_{\alpha+0210,0} c_{\alpha+1020,0}}\\
&=\frac{c_{t+e_0-e_1}c_{t+e_1-e_2}c_{t+e_2-e_0}}{c_{t+e_0-e_2}c_{t+e_1-e_0}c_{t+e_2-e_1}}\\
&=\prod_{t\in \mathrm{face}(v_0v_1v_2)}c^{\varepsilon_\circlearrowleft^{v_0v_1v_2}}_{t+v_0-v_1},
\end{aligned}
\end{equation}
where the last equality follows from~\eqref{eq:XExample}.
\end{proof}

\begin{lemma}\label{lemma:XandDiamonds} One can express the $X$-coordinates in terms of diamond coordinates: 
\begin{equation}
X_t=\left(\frac{d^{v_0v_1v_2}_{t-v_0-v_1}}{d^{v_0v_1v_2}_{t-v_0-v_2}}\right)^{\varepsilon_\circlearrowleft^{v_0v_1v_2}}
\end{equation}
Here $t$ is a face point spanned by $v_0$, $v_1$ and $v_2$. 
\end{lemma}
\begin{proof}
It is enough to prove this for $t=(t_0,t_1,t_2,0)$. Since
\begin{equation}
d^{012}_{t-e_0-e_1}=-\frac{c_{t+e_0-e_1}c_{t+e_2-e_0}}{c_tc_{t+e_2-e_1}},\qquad d^{012}_{t-e_0-e_2}=-\frac{c_{t+e_0-e_2}c_{t+e_1-e_0}}{c_{t+e_1-e_2}c_t},
\end{equation}
the result follows from~\eqref{eq:XExample}.
\end{proof}

\begin{figure}[htb]
\begin{center}
\begin{minipage}[b]{0.45\textwidth}
\begin{center}
\scalebox{0.7}{
\input{figures_gen/XCoords.tex}}
\end{center}
\end{minipage}
\hfill
\begin{minipage}[b]{0.45\textwidth}
\begin{center}
\scalebox{0.7}{
\input{figures_gen/diamondCoordinatesCancellationXCoords.tex}}
\end{center}
\end{minipage}\\
\begin{minipage}[t]{0.45\textwidth}
\caption{Exponents of the Ptolemy coordinates involved in the $X$-coordinate at $t$.}
\end{minipage}
\begin{minipage}[t]{0.45\textwidth}
\begin{center}
\caption{An $X$-coordinate as a quotient of two diamond coordinates.}
\end{center}
\end{minipage}
\end{center}
\end{figure}

\subsection{From natural $(\SL(n,\C),N)$-cocycles to natural $(\PGL(n,\C),B,H)$-cocycles.}

The natural map $\pi\colon\SL(n,\C)\to\PGL(n,\C)$ induces a map from decorations by $N$-cosets to decorations by $B$-cosets. Given a generic $\SL(n,\C)/N$-decoration $D$ on $\Delta_n^3$, we show how the natural $(\PGL(n,\C),B,H)$-cocycle $\mathcal L_{\alpha\beta\gamma}(\pi(D))$ can be obtained from the natural cocycle $\mathcal L_{\alpha\beta}(D)$ by the coboundary action \eqref{eq:CoboundaryAction} of an explicit coboundary given in terms of the diamond coordinates. This defines the map $\tau$ in diagram~\eqref{eqn:DiagramIntro} and gives rise to an explicit formula for $\mathcal L_{\alpha\beta\gamma}(\pi(D))$ in terms of the shapes.

Given an $\SL(n,\C)/N$-decoration $D$ with diamond coordinates $d^{v_0,v_1,v_2}_{1,i}$ consider the $0$-cochain on $\overline{\overline{\Delta^3_n}}$ given by
\begin{equation}
\tau^{v_0v_1v_2}(D) = \begin{pmatrix}\prod_{i=1}^{n-1} d^{v_0v_1v_2}_{1,i} \\ & \prod_{i=2}^{n-1} d^{v_0v_1v_2}_{1,i} \\ & & \ddots \\ & & & d^{v_0v_1v_2}_{1,n-i} \\ & & & & 1\end{pmatrix}=\prod_{i=1}^{n-1}H_i(d^{v_0v_1v_2}_{1,i}).
\end{equation}
We shall make use of the abbreviations
\begin{equation}
X_{k,i}=X_{kv_2+iv_0+(n-k-i)v_1}^{v_0v_1v_2},\qquad z_i=z_{(i-1)v_0+(n-1-i)v_1}^{v_0+v_1}.
\end{equation}

\begin{theorem}\label{thm:NaturalCocycleShapes}
Let $D$ be a generic $\PGL(n,\C)/B$-decoration of $\Delta^3_n$. The natural cocycle $\mathcal L_{\alpha\beta\gamma}(D)$ is given by
\begin{equation}\label{eq:NaturalCocycleFormula}
\begin{gathered}
\alpha^{v_0v_1v_2}=q_1,\qquad \beta^{v_0v_1v_2}=\prod_{k=1}^{n-1}\left(\prod_{i=1}^{n-k}x_i(1)\prod_{i=1}^{n-k-1}H_i(X_{k,i}^{\varepsilon^{v_0v_1v_2}_\circlearrowleft})\right)d_{\pm 1}\\
\gamma^{v_0v_1v_2}=\prod_{i=1}^{n-1}H_i(z_{i}^{-\varepsilon_\circlearrowleft^{v_0v_1v_2}}).
\end{gathered}
\end{equation}
Moreover, if $\widetilde D$ is any $\SL(n,\C)/N$-decoration lifting $D$, $\mathcal L(D)=\tau \mathcal L_{\alpha\beta}(\widetilde D)$.\qed
\end{theorem}

Before embarking on the proof, we give some examples.

\begin{example}For $n=2$,
\begin{equation}
\beta^{012}=x_1(1)d_{\pm 1}=\begin{pmatrix}-1&1\\&1\end{pmatrix},\qquad \gamma^{012}=H_1(z_1^{-1})=\begin{pmatrix}z_1^{-1}&\\&1\end{pmatrix}.
\end{equation}
For $n=3$, we have
\begin{equation}
\begin{gathered}
\beta^{012}=x_1(1)x_2(1)H_1(X_{1,1})x_1(1)d_{\pm 1}=
\begin{pmatrix}X_{1,1}&-X_{1,1}&1\\&-1&1\\&&1\end{pmatrix},
\\\gamma^{012}=
H_1(z_1^{-1})H_2(z_2^{-1})=\begin{pmatrix}z_1^{-1}z_2^{-1}&&\\&z_2^{-1}&\\&&1\end{pmatrix}.
\end{gathered}
\end{equation}
For $n=4$,
\begin{equation}
\begin{gathered}
\beta^{012}=x_1(1)x_2(1)x_3(1)H_1(X_{1,1})H_2(X_{1,2})x_1(1)x_2(1)H_1(X_{2,1})x_1(1)d_{\pm 1}\\ \gamma^{012}=d(z_1^{-1}z_2^{-1}z_3^{-1},z_2^{-1}z_3^{-1},z_3^{-1},1).
\end{gathered}
\end{equation}
\end{example}

\begin{remark}
The formula for $\beta^{v_0v_1v_2}$ is inspired by \cite[(9.14)]{FockGoncharov}.
\end{remark}

\begin{remark}\label{rm:DiagonalEntries}
Note that the diagonal entries of $\beta^{v_0v_1v_2}$ and $\gamma^{v_0v_1v_2}$ are given by
\begin{equation}
\beta^{v_0v_1v_2}_{ll}=(-1)^{n-l}\prod_{i=l}^{n-2}\prod_{k=1}^{n-1-i}X_{k,i}^{\varepsilon_\circlearrowleft^{v_0v_1v_2}},\quad \gamma_{ll}^{v_0v_1v_2}=\prod_{i=l}^{n-1}z_i^{-\varepsilon_\circlearrowleft^{v_0v_1v_2}}.
\end{equation}
\end{remark}

\begin{proof}[Proof of Theorem~\ref{thm:NaturalCocycleShapes}]
We prove that $\tau(\widetilde D)\mathcal L_{\alpha\beta}(\widetilde D)$ is given by~\eqref{eq:NaturalCocycleFormula}.
Since the last column of $\prod_{i=1}^{n-1}x_i(1)$ consists of $1$'s and since none of the other terms affect the last column, the middle edges are thus normalized, so by Lemma~\ref{lemma:UniversalProperty}, the cocycle is indeed the natural cocycle of $D$. 

Let $\widetilde \alpha$, $\widetilde \beta$ (and $\widetilde \gamma=\id$) be the labelings, of long, middle and short edges given by $L_{\alpha\beta}(\widetilde D)$. By Proposition~\ref{prop:NaturalCocycleLongShort} these are given by \eqref{eq:betaTrunc}. Let $\tau=\tau(\widetilde D)$.

\textit{Long edges:}
We must prove that $(\tau^{v_0v_1v_2})^{-1}\widetilde\alpha^{v_0v_1v_2}\tau^{v_1v_0v_2}=q_1$.
Letting 
\begin{equation}
l_k=\prod_{i=k}^{n-1}d^{v_0v_1v_2}_{1,i},\quad m_k=(-1)^{k-1}e^{v_0v_1}_{(n-k)v_0+(k-1)v_1},\quad r_k=\prod_{i=k}^{n-1}d^{v_1v_0v_2}_{1,i},
\end{equation}
this is equivalent to proving that
\begin{equation}
d(l_1,\dots,l_n)^{-1}q(m_1,\dots,m_n)d(r_1,\dots,r_n)=q_1\in\PGL(n,\C).
\end{equation}

Hence, we must prove that $l_{n-k+1}^{-1}m_kr_k$ is independent of $k$. From Figure~\ref{fig:Cancellations} it follows that
\begin{equation}
\prod_{i=k}^{n-1}d^{v_0v_1v_2}_{1,i}=\varepsilon^{n-k}\frac{c_{nv_0}c_{(k-1)v_0+(n-k)v_1+v_2}}{c_{(n-1)v_0+v_2}c_{kv_0+(n-k)v_1}},\quad \prod_{i=k}^{n-1}d^{v_1v_0v_2}_{1,i}=(-\varepsilon)^{n-k}\frac{c_{nv_1}c_{(k-1)v_1+(n-k)v_0+v_2}}{c_{(n-1)v_1+v_2}c_{kv_1+(n-k)v_1}}
\end{equation}
where $\varepsilon=-\varepsilon_<^{v_0,v_1,v_2}$. Hence, we have
\begin{equation}
\begin{gathered}
l_{n-k+1}=\varepsilon^{k-1}\frac{c_{(n-k)v_0+(k-1)v_1+v_2}}{c_{(n-k+1)v_0+(k-1)v_1}c_{(n-1)v_0+v_2}},\quad m_k=(-1)^{k-1}\frac{c_{(n-k)v_0+kv_1}}{c_{(n-k+1)v_0+(k-1)v_1}}\\
r_k=(-\varepsilon)^{n-k}\frac{c_{(k-1)v_1+(n-k)v_0+v_2}}{c_{(n-1)v_1+v_2}c_{kv_1+(n-k)v_0}},
\end{gathered}
\end{equation}
from which it follows that
\begin{equation}
l_{n-k+1}^{-1}m_kr_k=(-\varepsilon)^{n-1}\frac{c_{(n-1)v_0+v_2}}{c_{(n-1)v_1+v_2}},
\end{equation}
which is independent of $k$. This proves the result.

\begin{figure}[htb]
\begin{center}
\scalebox{0.75}{
\input{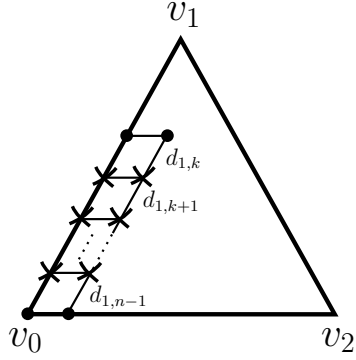}}
\end{center}
\caption{Cancellations.}\label{fig:Cancellations}
\end{figure}

\textit{Middle edges:} We must prove that $\tau^{-1}_{v_0v_1v_2}\widetilde\beta_{v_0v_1v_2}\tau_{v_0v_2v_1}=\beta_{v_0v_1v_2}$.
Using the basic commutator relations
\begin{equation}\label{eq:CommutatorRelations}
H_i(x)x_j(y)=x_j(y)H_i(x) \text{ if } i\neq j, \quad x_i(y)=H_i(y)x_i(1)H_i(y)^{-1}
\end{equation}
the expression
\begin{equation}
\widetilde\beta^{v_0v_1v_2}=\prod_{k=1}^{n-1}\prod_{i=1}^{n-k}x_i(d_{k,i})=\prod_{k=1}^{n-1}\prod_{i=1}^{n-k}H_i(d_{k,i})x_i(1)H_i(d_{k,i})^{-1}
\end{equation}
expands to
\begin{equation}\label{eq:AuxMiddle1}
\widetilde\beta^{v_0v_1v_2}=\prod_{k=1}^{n-1}\left(\prod_{i=1}^{n-k}H_i(d_{k,i})\prod_{i=1}^{n-k}x_i(1)\left(\prod_{i=1}^{n-k-1}H_i(d_{k,i})^{-1}\right) H_{n-k}(d_{k,n-k})^{-1}\right).
\end{equation}
We have for brevity omitted the superscript $v_0v_1v_2$ of the diamond coordinates. Letting $$\mathcal H_k=\prod_{i=1}^{n-k}H_i(d_{k,i})\quad\mbox{and}\quad \mathcal H'_k=\prod_{i=1}^{n-k-1}H_i(d_{k,i}),$$
and moving the terms $H_{n-k}(d_{k,n-k})^{-1}$ to the right, we have
\begin{equation}
\widetilde\beta^{v_0v_1v_2}=\prod_{k=1}^{n-1}\left(\mathcal H_k\left(\prod_{i=1}^{n-k}x_i(1)\right)\mathcal H'^{-1}_k\right) \prod_{k=1}^{n-1}  H_{n-k}(d_{k,n-k})^{-1}.
\end{equation}
Since $\mathcal H_n=1$, the product $(\mathcal H_1\cdots \mathcal H'^{-1}_1)(\mathcal H_2 \cdots \mathcal H'^{-1}_2)\cdots$ equals $\mathcal H_1 (\cdots \mathcal H'^{-1}_1 \mathcal H_2)(\cdots \mathcal H'^{-1}_2 \mathcal H_3)\cdots$, and we obtain 
\begin{equation}
\widetilde\beta_{v_0v_1v_2}=\mathcal{H}_1\prod_{k=1}^{n-1}\left(\left(\prod_{i=1}^{n-k}x_i(1)\right)\mathcal H'^{-1}_k \mathcal H_{k+1}\right) \prod_{k=1}^{n-1}  H_{n-k}(d_{k,n-k})^{-1}.
\end{equation}
Using \eqref{eq:DiamondDef}, we have
\begin{equation}\label{eq:AuxMiddle2}
d_{1,i}^{v_0v_2v_1}=d_{(i-1)v_0+(n-1-i)v_2}^{v_0v_2v_1}=-d_{(i-1)v_0+(n-1-i)v_2}^{v_0v_1v_2}=-d_{n-i,i}^{v_0v_1v_2},
\end{equation}
and since (last equality follows from~\eqref{eq:AuxMiddle2})
\begin{equation}
\tau^{v_0v_1v_2}=\prod_{i=1}^{n-1}H_i(d^{v_0v_1v_2}_{1,i})=\mathcal H_1,\qquad \tau^{v_0v_2v_1}=\prod_{i=1}^{n-1}H_i(d^{v_0v_2v_1}_{1,i})=\prod_{k=1}^{n-1}H_{n-k}(-d_{k,n-k}).
\end{equation}
we have
\begin{equation}
\beta_{v_0v_1v_2}=\tau^{-1}_{v_0v_1v_2}\widetilde\beta_{v_0v_1v_2}\tau_{v_0v_2v_1} = \prod_{k=1}^{n-1}\left(\left(\prod_{i=1}^{n-k}x_i(1)\right)\mathcal H'^{-1}_k \mathcal H_{k+1}\right)d_{\pm 1}.
\end{equation}
By Lemma~\ref{lemma:XandDiamonds}, $\frac{d_{k+1,i}}{d_{k,i}}=X_{k,i}^{\varepsilon_\circlearrowleft(v_0v_1v_2)}$, so that
\begin{equation}
\mathcal H'^{-1}_k \mathcal H_{k+1} = \prod_{i=1}^{n-k} H_i(d_{k,i})^{-1} H_i(d_{k+1,i}) = \prod_{i=1}^{n-k} H_i(X_{k,i}^{\varepsilon_\circlearrowleft(v_0v_1v_2)}).
\end{equation}
This proves the result.

\textit{Short edges:} We must prove that $(\tau^{v_0v_1v_2})^{-1}\tau^{v_0v_1v_3}=\gamma^{v_0v_1v_2}$. We have
\begin{equation}
(\tau^{v_0v_1v_2})^{-1}\tau^{v_0v_1v_3}=\left(\prod_{i=1}^{n-1}H_i(d^{v_0v_1v_2}_{1,i})\right)^{-1}\prod_{i=1}^{n-1}H_i(d^{v_0v_1v_3}_{1,i})=\prod_{i=1}^{n-1}H_i(\frac{d^{v_0v_1v_3}_{1,i}}{d^{v_0v_1v_2}_{1,i}}).
\end{equation}
The result now follows from Lemma~\ref{ShapesandDiamonds} below.

\begin{lemma}\label{ShapesandDiamonds}
The shape parameters in \eqref{eq:NaturalCocycleFormula} are given in terms of diamond coordinates:
\begin{equation}
z_{i}^{-\varepsilon^{v_0v_1v_2}_\circlearrowleft}=\frac{d^{v_0v_1v_3}_{1,i}}{d^{v_0v_1v_2}_{1,i}}
\end{equation}
\end{lemma}

\begin{proof}
Let $\alpha=(i-1)v_0+(n-1-i)v_1$, so that $z_i=z^{v_0+v_1}_{\alpha}$ and $d^{v_0v_1v_k}_{1,i}=d^{v_0v_1v_k}_\alpha$, $k=2,3$. By compatibility under reordering, it is enough to prove the result for $v_i=i$. We have
\begin{equation}
z^{1100}_{\alpha}=\frac{c_{\alpha+e_0+e_3}c_{\alpha+e_1+e_2}}{c_{\alpha+e_0+e_2}c_{\alpha+e_1+e_3}},\qquad d^{01k}_\alpha=-\frac{c_{\alpha+2e_0}c_{\alpha+e_1+e_k}}{c_{\alpha+e_0+e_1}c_{\alpha+e_0+e_k}}.
\end{equation}
Hence, $z_i^{-1}=d^{013}_{1,i}/d^{012}_{1,i}$ proving the result.
\end{proof}
This concludes the proof of Theorem~\ref{thm:NaturalCocycleShapes}.
\end{proof}

\begin{remark}\label{rm:NaturalLift}
Theorem~\ref{thm:NaturalCocycleShapes} implies that for a generic $\PGL(n,\C)/B$-decoration on $(M,\T)$, the restriction of the natural cocycle to $\partial M$ has a canonical lift to a cocycle with values in $B\subset\GL(n,\C)$ (not just in $\PGL(n,\C)$).
\end{remark}

\section{From shape assignments to cocycles}\label{sec:ReconstructRep}

We now prove that the bottom row of diagram~\eqref{eqn:DiagramIntro} consists of one-one correspondences.
The idea is to first prove that a shape assignment determines a natural cocycle on each doubly truncated simplex. This is a consequence of the internal gluing equations. The face and edge equations imply that the cocycles glue together to a cocycle on $M$; the middle edges glue together because of the face equations, and the edge equations imply that we can fill in the prisms.

\begin{lemma}\label{lemma:AgreeonFaces}
If two shape assignments $z$ and $w$ agree on two faces $i$ and $j$, then $z=w$, i.e.~if $z_s^e=w_s^e$ when $e_i=s_i=0$ or $e_j=s_j=0$, then $z_s^e=w_s^e$ for all $(s,e)\in\Delta^3_{n-2}(\Z)\times\dot\Delta^3_2(\Z)$.
\end{lemma}

\begin{proof}
We may assume that $z$ and $w$ agree on face $2$ and $3$. It is enough to prove that $z_s^{1100}=w_s^{1100}$ for all $s=(s_0,s_1,s_2,s_3)\in\Delta^3_{n-2}(\Z)$. By assumption, this holds if either $s_2$ or $s_3$ is $0$. Suppose by induction that $z_s^{1100}=w_s^{1100}$ for all $s$ with $s_2+s_3<k$, and let $s\in\Delta^3_{n-2}(\Z)$ be a subsimplex with $s_2+s_3=k$. Since the result holds, when either $s_2$ or $s_3$ is $0$, we may assume that $s=\alpha+0011$, with $\alpha\in\Delta^3_{n-4}(\Z)$. Let $t=s+1100$. By Lemma~\ref{lemma:InternalEquation}, $z$ and $w$ satisfy the internal gluing equations, i.e.~we have
\begin{equation}
z^{0011}_{\alpha+1100}z^{0101}_{\alpha+1010}z^{0110}_{\alpha+1001}z^{1001}_{\alpha+0110}z^{1010}_{\alpha+0101}z^{1100}_s=w^{0011}_{\alpha+1100}w^{0101}_{\alpha+1010}w^{0110}_{\alpha+1001}w^{1001}_{\alpha+0110}w^{1010}_{\alpha+0101}w^{1100}_s,
\end{equation}
which equals $1$. Note that for all terms except $z_s^{1100}$ and $w_s^{1100}$, the lower index satisfies $s_2+s_3<k$. By induction, each term $z^{1111-I}_{\alpha+I}$ equals $w^{1111-I}_{\alpha+I}$. Hence $z^{1100}_s=w_s^{1100}$, completing the induction. 
\end{proof}

\begin{lemma}\label{lemma:UniqueFactorization}
The factorization of the middle edges is unique, i.e.~if
\begin{equation}\label{eq:UniqueFactorization}
\prod_{k=1}^{n-1}\left(\prod^{n-k}_{i=1}x_i(1)\prod_{i=1}^{n-k-1}H_i(a_{k,i})\right)=\prod_{k=1}^{n-1}\left(\prod^{n-k}_{i=1}x_i(1)\prod_{i=1}^{n-k-1}H_i(b_{k,i})\right),
\end{equation}
then $a_{k,i}=b_{k,i}$ for all $k,i$.
\end{lemma}

\begin{proof}
Suppose~\eqref{eq:UniqueFactorization} holds. In particular, all diagonal entries are equal. Hence, as in Remark~\ref{rm:DiagonalEntries}, the equality $\prod_{i=l}^{n-2}\prod_{k=1}^{n-1-i} a_{k,i}=\prod_{i=l}^{n-2}\prod_{k=1}^{n-1-i} b_{k,i}$ holds for all $l=1,\dots,n$. The result now follows by induction.
\end{proof}

\begin{proposition}\label{prop:ZisSurjective}
The map $\mathcal Z$ from generic decorations on $\Delta^3_n$ to shape assignments on $\Delta^3_n$ is surjective.
\end{proposition}

\begin{proof}
Let $z$ be a shape assignment. We wish to construct a decoration $D$ with $\mathcal Z(D)=z$. Let $D=(B, q_1B, \beta_{012} q_1 B, \gamma_{012} \beta_{013} q_1 B)$ and let $z'=\mathcal Z(D)$. By Lemma~\ref{lemma:AgreeonFaces}, it is enough to prove that $z'$ agrees with $z$ on face $2$ and $3$. 
We prove this for face $3$ ($s_3=0$), face $2$ being similar. We use induction on $s_2$. Let $\beta'$ and $\gamma'$ denote the labelings of the natural cocycle of $D$. Let $X_t$ and $X'_t$ denote the $X$ coordinates of $z$ and $z'$. Note that $\beta'_{012}=\beta_{012}$ and $\gamma'_{012}=\gamma_{012}$. Since $\gamma'_{012}=\gamma_{012}$, the equality $z_s^{1100}=w_s^{1100}$ holds for $s_2=0$ proving the induction start. Since $\beta'_{012}=\beta_{012}$, it follows from 
Lemma~\ref{lemma:UniqueFactorization} that $X_t=X'_t$ for all $t$ on face $2$ and $3$.
Now suppose by induction that $z_s^{1100}=w_s^{1100}$ holds for $s_2<k$. Let $t=s+1100$.  
We have
\begin{equation}
-z_s^{1100}z_{t-1010}^{1010}z^{0110}_{t-0110}=X_t=X'_t=-z_s'^{1100}z_{t-1010}'^{1010}z'^{0110}_{t-0110}.
\end{equation}
By induction, $z_{t-1010}^{1010}=z_{t-1010}'^{1010}$ and $z_{t-0110}^{0110}=z_{t-0110}'^{0110}$, so we must also have
$z_s^{1100}=z_s'^{1100}$. This proves the result.
\end{proof}

\begin{theorem}\label{thm:One-One}
The bottom row of diagram~\eqref{eqn:DiagramIntro} consists of one-one correspondences. 
\end{theorem}

\begin{proof}
We first prove this for a simplex. The map $\mathcal L_{\alpha\beta\gamma}$ is bijective by Remark~\ref{rm:CocycleToDecoration}. Injectivity of $\mathcal Z$ follows from Theorem~\ref{thm:NaturalCocycleShapes}, and surjectivity was proved in Proposition~\ref{prop:ZisSurjective}. Now suppose $z$ is a shape assignment on $(M,\T)$. We must prove that $z$ determines a generic decoration, or equivalently a natural cocycle. By Proposition~\ref{prop:ZisSurjective} $z$ determines a natural cocycle on each doubly truncated simplex. We must prove that these fit together to form a cocycle on $M$. The labelings of long edges obviously match up, and by \eqref{eq:NaturalCocycleFormula} and Remark~\ref{rm:HalfofFaceEq}, the middle edges match up if and only if the face equations are satisfied. 
Now all that is left to prove is that the induced labeling on the prisms are cocycles. This is a direct consequence of the edge equations.
\end{proof}

\begin{remark} 
It follows from Theorem~\ref{thm:One-One} that \eqref{eq:NaturalCocycleFormula} gives an explicit map from shape assignments to cocycles. Combined with Remark~\ref{rm:CocycleToDecoration} this gives explicit inverses of $\mathcal Z$ and $\mathcal L_{\alpha\beta\gamma}$. Similarly, we have explicit inverses of $\mathcal C$ and $\mathcal L_{\alpha\beta}$ via \eqref{eq:betaTrunc}.
\end{remark}

\section{Duality}\label{sec:Duality}

In this section we make some observations about the relationship between the shape coordinates and the Ptolemy coordinates. Our observations suggest that there is a fundamental duality between the two sets of coordinates, which is interesting in its own right. We believe that this duality is a $3$-dimensional aspect of the duality (see Fock-Goncharov~\cite[p.~33]{FockGoncharov}) between $\mathcal A$-coordinates and $\mathcal X$-coordinates on higher Teichm\"{u}ller space of a simply connected Lie group (e.g.~$\SL(n,\C)$), respectively, its Langlands dual group (e.g.~$\PGL(n,\C)$).

Note that for each subsimplex, one shape parameter determines the other two. We single out one:

\begin{definition}
We call the parameters $z^{1100}_s$ \emph{shape coordinates}.
\end{definition}

As is customary for $n=2$, we can write the gluing equations entirely in terms of the shape coordinates. Let $\sub_n(\T)$ denote the set of all subsimplices of the simplices of $\T$.

\begin{observation}[Duality] The coordinates and their relations are parametrized by the following sets:
\begin{center}
\begin{tabular}{c|c}
Ptolemy coordinates & shape coordinates\\
$\dot\T_n(\Z)$ & $\sub_n(\T)$\\ \hline
Ptolemy relations & gluing equations\\
$\sub_n(\T$) & $\dot\T_n(\Z)$
\end{tabular}
\end{center}
In particular, we have
\begin{equation}\label{eq:Duality}
\begin{array}{rcl}
\#\mbox{\{Ptolemy coordinates\}} & = & \#\mbox{\{Gluing equations\}}\\
\#\mbox{\{Ptolemy relations\}} & = & \#\mbox{\{Shape coordinates\}}.
\end{array}
\end{equation}
\end{observation}

\begin{proposition}\label{prop:DualityTori}
If all boundary components of $M$ are tori, we have
\begin{equation}
\#\left\{\txt{Ptolemy\\ ~coord.}\right\}=\#\left\{\txt{Ptolemy\\relations}\right\}=\#\left\{\txt{Shape\\ ~coord.}\right\} =\#\left\{\txt{Gluing\\equations}\right\}=\binom{n+1}{3}t,
\end{equation}
where $t$ is the number of simplices of $\T$.
\end{proposition}
\begin{proof}
This follows immediately from Lemma~\ref{lemma:Rank}.
\end{proof}




\section{The cusp equations}\label{sec:CuspEquations}

The decomposition of $M$ into doubly truncated simplices and prisms induces a polyhedral decomposition of $\partial M$. Note that every simple closed curve in $\partial M$ is isotopic to an edge path in this decomposition. 

Let $\sigma$ be a natural $(\PGL(n,\C),B,H)$-cocycle on $M$. By Remark~\ref{rm:NaturalLift}, the restriction of $\sigma$ to $\partial M$ has a canonical lift (also denoted by $\sigma$) to a cocycle with values in $B\subset\GL(n,\C)$.  

\begin{definition}
Let $\lambda$ be a closed edge path in $\partial M$ and let $\beta_1,\dots\beta_r$ and $\gamma_1,\dots,\gamma_s$ be the labelings induced by $\sigma$ of the middle, respectively, short edges traversed by $\lambda$. For $l=1,\dots, n-1$, the \emph{level $l$ cusp equation} of $\lambda$ is the equation
\begin{equation}
\prod_{j=1}^r(\beta_j)_{ll}\prod_{j=1}^s(\gamma_j)_{ll}=1,
\end{equation}
where the subscript $ll$ denotes the $l$th diagonal entry.
\end{definition}

\begin{figure}[htb]
\begin{center}
\scalebox{0.9}{
\input{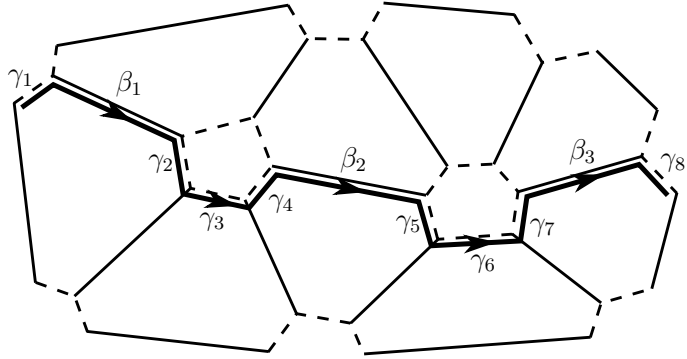}}
\end{center}
\caption{A curve in the polyhedral decomposition of $\partial M$.}
\end{figure}

\begin{lemma}
The representation $\rho$ determined (up to conjugation) by $\sigma$ is boundary-unipotent if and only if the cusp equation at each level is satisfied for each edge path representing a generator of the fundamental group of a boundary component of $\partial M$.
\end{lemma}

\begin{proof}
By definition, $\rho$ is boundary-unipotent if and only if for each closed edge path $\lambda$ in $\partial M$, the product of the labelings of edges traversed by $\lambda$ is in $N$. This proves the result.
\end{proof}

\begin{remark}
Note that for $n=2$ we recover the traditional cusp equations.
\end{remark}

\subsection{Simplifying the cusp equations}

\begin{lemma} 
The cusp equations are equivalent to the equations
\begin{equation} \label{eq:SimplifiedCuspEquation}
\prod_{j=1}^r\frac{(\beta_j)_{ll}}{(\beta_j)_{l+1,l+1}}\prod_{j=1}^s\frac{(\gamma_j)_{ll}}{(\gamma_i)_{l+1,l+1}}=1.
\end{equation}
Moreover, each factor is given by an expression of the form
\begin{equation}\label{eq:Simplification}
\frac{(\beta_j)_{ll}}{(\beta_j)_{l+1,l+1}}=-\prod_{k=1}^{n-1-l}X_{k,l,\Delta}^{\varepsilon_\circlearrowleft^{v_0v_1v_2}},\qquad \frac{(\gamma_j)_{ll}}{(\gamma_j)_{l+1,l+1}}=z_{l,\Delta}^{-\varepsilon_\circlearrowleft^{v_0v_1v_2}}
\end{equation}
where $v_0v_1v_2$ is the starting vertex of $\beta_j$, respectively, $\gamma_j$ and $\Delta$ is the corresponding simplex.
\end{lemma}

\begin{proof}
Equation \eqref{eq:SimplifiedCuspEquation} follows from the fact that the $n$th diagonal entry is $1$ for both short and middle edges.
By Remark~\ref{rm:DiagonalEntries}, the diagonal entries are expressions of the form
\begin{equation}
\label{main}
(\beta_j)_{ll}=(-1)^{n-l}\prod_{i=l}^{n-2}\prod_{k=1}^{n-1-i}X_{k,i}^{\varepsilon_\circlearrowleft^{v_0v_1v_2}},\quad (\gamma_j)_{ll}=\prod_{i=l}^{n-1}z_i^{-\varepsilon_\circlearrowleft^{v_0v_1v_2}}.
\end{equation}
Taking quotients, this proves the result.
\end{proof}

\begin{figure}[htb]
\begin{center}
\scalebox{0.9}{
\input{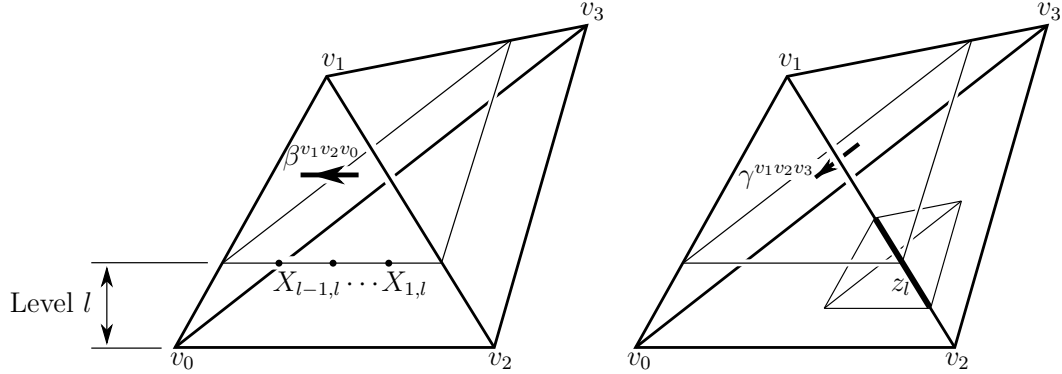}}
\end{center}
\caption{The terms in the simplified cusp equations at level $l$.}\label{fig:CuspEqTerms}
\end{figure}

\begin{remark}
Note that the contribution from a middle edge is (minus) the product of the $X$-coordinates at level $l$. The contribution from a short edge is the shape coordinate at level $l$. See Figure~\ref{fig:CuspEqTerms}. 
\end{remark}

\begin{remark}
Note that the cusp equations can be written in the form~\eqref{eq:GluinEqSimplified}, but with $1$ replaced by a sign.
\end{remark}

\section{Example: the figure-eight knot}\label{sec:Example}

Consider the triangulation of the figure-eight knot complement given in Figure~\ref{fig:Fig8Ordered}. Figure~\ref{fig:Figure8Cusp} shows the induced triangulation of the link of the ideal vertex and indicates two peripheral curves $\mu$ and $\lambda$ generating the peripheral fundamental group. These are \emph{not} the standard meridian and longitude of the knot. The shape, respectively, $X$-coordinates of the left simplex are denoted by $z^e_s$ and $X_t$, whereas those for the right simplex are denoted by $w_s^e$ and $Y_t$.

\begin{figure}[htb]
\begin{center}
\scalebox{0.85}{
\input{figures_gen/Figure8KnotCusp.tex}}
\end{center}
\caption{Generators of the peripheral fundamental group of the figure-eight knot complement. The indicated coordinates are those involved in the cusp equations at level $l=2$ for $n=4$.}\label{fig:Figure8Cusp}
\end{figure}

We first consider the gluing equations for $n=3$. By examining Figure~\ref{fig:Fig8Ordered}, we see that there are $4$ edge points giving rise to the gluing equations

\begin{equation}
\begin{aligned}
z^{0101}_{0100}z^{0110}_{0100}z^{1010}_{1000}(w^{1100}_{1000})^{-1}(w^{1001}_{1000})^{-1}(w^{0011}_{0010})^{-1}&=1,\\
z^{0101}_{0001}z^{0110}_{0010}z^{1010}_{0010}(w^{1100}_{0100})^{-1}(w^{1001}_{0001})^{-1}(w^{0011}_{0001})^{-1}&=1,\\
z^{1100}_{1000}z^{1001}_{1000}z^{0011}_{0010}(w^{0101}_{0100})^{-1}(w^{0110}_{0100})^{-1}(w^{1010}_{1000})^{-1}&=1,\\
z^{1100}_{0100}z^{1001}_{0001}z^{0011}_{0001}(w^{0101}_{0001})^{-1}(w^{0110}_{0010})^{-1}(w^{1010}_{0010})^{-1}&=1,
\end{aligned}
\end{equation}
and four face points giving rise to the equations
\begin{equation}
\begin{aligned}
z^{1100}_{0010}z^{0110}_{1000}z^{1010}_{0100}(w^{0011}_{1000})^{-1}(w^{1001}_{0010})^{-1}(w^{1010}_{0001})^{-1}&=1,\\
z^{0101}_{1000}z^{1001}_{0100}z^{1100}_{0001}(w^{0011}_{0100})^{-1}(w^{0101}_{0010})^{-1}(w^{0110}_{0001})^{-1}&=1,\\
z^{0011}_{0100}z^{0101}_{0010}z^{0110}_{0001}(w^{0101}_{1000})^{-1}(w^{1001}_{0100})^{-1}(w^{1100}_{0001})^{-1}&=1,\\
z^{0011}_{1000}z^{1001}_{0010}z^{1010}_{0001}(w^{0110}_{1000})^{-1}(w^{1010}_{0100})^{-1}(w^{1100}_{0010})^{-1}&=1.
\end{aligned}
\end{equation}
The cusp equations for $\lambda$ are
\begin{equation}
\begin{aligned}
z^{0101}_{0001}X_{0111}w^{1100}_{0100}Y_{1110}z^{1001}_{0001}X_{1101}w^{0110}_{0010}Y_{1110}z^{1010}_{1000}X_{1110}w^{0011}_{0010}Y_{0111}z^{1001}_{1000}X_{1011}w^{0110}_{0100}Y_{0111}&=1,\\
z^{0101}_{0100}w^{1100}_{1000}z^{1001}_{1000}w^{0110}_{0100}z^{1010}_{0010}w^{0011}_{0001}z^{1001}_{0001}w^{0110}_{0010}&=1,
\end{aligned}
\end{equation}
and the cusp equations for $\mu$ are
\begin{equation}
\begin{aligned}
z^{1010}_{0010}X_{1011}w^{1010}_{0010}Y_{1011}&=1,\\
z^{1010}_{1000}w^{1010}_{1000}&=1.
\end{aligned}
\end{equation}

Using Magma~\cite{Magma} to compute the primary decomposition of the ideal generated by the above equations (together with the shape parameter relations~\eqref{eqn:shapeparamrel}, the formula \eqref{eq:XandProducts} for the $X$-coordinates in terms of the shapes, and an extra equation making sure that none of the shapes are $0$ and $1$) we obtain $4$ zero-dimensional algebraic components displayed below. For notational convenience, we write $z_i=z_{e_i}^{1100}$ (similarly for $w$).
\begin{equation}
\begin{aligned}
       z_0&=w_3+\frac{3}{2},\quad &z_1&=\frac{1}{2}w_3+\frac{1}{2},\quad &z_2&=-\frac{1}{2}w_3+\frac{1}{4},\quad &z_3= -w_3+1,\\
        w_0&=-w_3-\frac{1}{2},\quad &w_1&=2w_3+2,\quad &w_2&=-2w_3+1,\quad &w_3^2 + \frac{1}{2}w_3 + \frac{1}{2},
\end{aligned}
\end{equation}


\begin{equation}
\begin{aligned}
z_0&=-w_3+1\quad &z_1&=w_3,\quad&z_2&=w_3,\quad&z_3=-w_3+1,\\
w_0&=w_3,\quad&w_1&=-w_3+1,\quad&w_2&=-w_3+1,\quad&w_3^2-w_3+1,
\end{aligned}
\end{equation}


\begin{equation}
\begin{aligned}
z_0&=w_3-\frac{3}{2},\quad&z_1&=-2w_3+4,\quad&z_2&=2w_3-1,\quad&z_3=-w_3+1,\\
w_0&=-w_3+\frac{5}{2},\quad&w_1&=-\frac{1}{2}w_3+1,\quad&w_2&=\frac{1}{2}w_3-\frac{1}{4},\quad&w_3^2-\frac{5}{2}w_3+2,
\end{aligned}
\end{equation}


\begin{equation}
\begin{aligned}
z_0=z_1=z_2=z_3=-w_3+1,\qquad w_0=w_1=w_2=w_3,\qquad w_3^2-w_3+1.
\end{aligned}
\end{equation}


\begin{remark}
Note that the first and third component is defined over $\Q(\sqrt{-7})$, whereas the second and fourth are defined over $\Q(\sqrt{-3})$. The fourth component corresponds to the representation arising from the geometric representation via the canonical irreducible map $\PSL(2,\C)\to\SL(3,\C)$. The fact that for this component the shapes of all subsimplices are equal for both of the simplices is a general phenomenon, see Garoufalidis-Thurston-Zickert~\cite[Theorem~11.3]{GaroufalidisThurstonZickert}.
\end{remark}

\begin{remark}
All representations except the second component lift uniquely to boundary-unipotent representations in $\SL(3,\C)$, so these are also detected by the Ptolemy variety (the non-geometric representations were ignored in~\cite{GaroufalidisThurstonZickert}, since they have $0$ volume). 
\end{remark}

\begin{remark}
The Ptolemy varieties seem to be much better suited for exact computations. For $n=2$ exact computations of Ptolemy varieties are usually very fast when there are less than $15$ simplices (usually a fraction of a section on a laptop). In comparison, exact computations of the gluing equation varieties require a lot more time and memory and are often impractical when there are more than a few simplices. Using the monomial map $\mu$, one can obtain solutions to the gluing equations from the Ptolemy coordinates.
\end{remark}

The gluing equations for $n=4$ are shown in Table~\ref{table:Figure8KnotEqns}. 

\begin{table}[p]
\caption{Gluing equations for the figure-eight knot and $n=4$.}\label{table:Figure8KnotEqns}
\begin{center}
\setlength\extrarowheight{5pt}
\begin{tabular}{c|c}
Face A: & Face B:\\ \hline \hline
$z_{2000}^{0110} z_{1100}^{1010} z_{1010}^{1100} (w_{2000}^{0011})^{-1}   (w_{1010}^{1001})^{-1} (w_{1001}^{1010})^{-1}  =1$ & 
$z^{0101}_{2000} z^{1001}_{1100} z^{1100}_{1001} (w^{0011}_{0200})^{-1} (w^{0101}_{0110})^{-1} (w^{0110}_{0101})^{-1} = 1$\\ \hline
$z_{1100}^{0110} z_{0200}^{1010} z_{0110}^{1100} (w_{1010}^{0011})^{-1}   (w_{0020}^{1001})^{-1} (w_{0011}^{1010})^{-1}  =1$ &
$z^{0101}_{1100} z^{1001}_{0200} z^{1100}_{0101} (w^{0011}_{0110})^{-1} (w^{0101}_{0020})^{-1} (w^{0110}_{0011})^{-1} = 1$\\ \hline
$z_{1010}^{0110} z_{0110}^{1010} z_{0020}^{1100} (w_{1001}^{0011})^{-1}   (w_{0011}^{1001})^{-1} (w_{0002}^{1010})^{-1}  =1$ &
$z^{0101}_{1001} z^{1001}_{0101} z^{1100}_{0002} (w^{0011}_{0101})^{-1} (w^{0101}_{0011})^{-1} (w^{0110}_{0002})^{-1} = 1$ \\ \hline
Face C: & Face D:\\ \hline \hline
$z^{0011}_{0200} z^{0101}_{0110} z^{0110}_{0101} (w^{0101}_{2000})^{-1} (w^{1001}_{1100})^{-1} (w^{1100}_{1001})^{-1} =1$ &
$z^{0011}_{2000} z^{1001}_{1010} z^{1010}_{1001} (w^{0110}_{2000})^{-1} (w^{1010}_{1100})^{-1} (w^{1100}_{1010})^{-1}=1$\\ \hline
$z^{0011}_{0110} z^{0101}_{0020} z^{0110}_{0011} (w^{0101}_{1100})^{-1} (w^{1001}_{0200})^{-1} (w^{1100}_{0101})^{-1} =1$ &
$z^{0011}_{1010} z^{1001}_{0020} z^{1010}_{0011} (w^{0110}_{1100})^{-1} (w^{1010}_{0020})^{-1} (w^{1100}_{0110})^{-1}=1$\\ \hline
$z^{0011}_{0101} z^{0101}_{0011} z^{0110}_{0002} (w^{0101}_{1001})^{-1} (w^{1001}_{0101})^{-1} (w^{1100}_{0002})^{-1} =1$ &
$z^{0011}_{1001} z^{1001}_{0011} z^{1010}_{0002} (w^{0110}_{1010})^{-1} (w^{1010}_{0002})^{-1} (w^{1100}_{0020})^{-1}=1$ \\ \hline
Edge $\rightarrow$ & Edge $\twoheadrightarrow$ \\ \hline \hline
$z_{2000}^{1010} z_{0200}^{0110} z_{0200}^{0101} (w_{2000}^{1100})^{-1} (w_{2000}^{1001})^{-1} (w_{0020}^{0011})^{-1} = 1$ &
$z_{2000}^{1100} z_{2000}^{1001} z_{0020}^{0011} (w_{2000}^{1010})^{-1} (w_{0200}^{0110})^{-1} (w_{0200}^{0101})^{-1}=1$ \\ \hline
$z_{1010}^{1010} z_{0110}^{0110} z_{0101}^{0101} (w_{1100}^{1100})^{-1} (w_{1001}^{1001})^{-1} (w_{0011}^{0011})^{-1}=1$ &
$z_{1100}^{1100} z_{1001}^{1001} z_{0011}^{0011} (w_{1010}^{1010})^{-1} (w_{0110}^{0110})^{-1} (w_{0101}^{0101})^{-1}=1$ \\ \hline
$z_{0020}^{1010} z_{0020}^{0110} z_{0002}^{0101} (w_{0200}^{1100})^{-1} (w_{0002}^{1001})^{-1} (w_{0002}^{0011})^{-1}=1$ &
$z_{0200}^{1100} z_{0002}^{1001} z_{0002}^{0011} (w_{0020}^{1010})^{-1} (w_{0020}^{0110})^{-1} (w_{0002}^{0101})^{-1}=1$ \\ \hline
Interior equation for $z$& Interior equation for $w$\\ \hline \hline
$z_{1100}^{0011}z^{0101}_{1010}z^{0110}_{1001}z^{1001}_{0110}z^{1010}_{0101}z^{1100}_{0011}=1$&$w_{1100}^{0011}w^{0101}_{1010}w^{0110}_{1001}w^{1001}_{0110}w^{1010}_{0101}w^{1100}_{0011}=1$
\end{tabular}
\end{center}

\begin{center}
\setlength\extrarowheight{5pt}
\begin{tabular}{llllll}
\multicolumn{4}{c}{Cusp equations for $\mu$} \\ \hline \hline
level $l=1$ &
$z^{1010}_{0020}X_{1021}X_{1012}$ & $w^{1010}_{0020}Y_{1021}Y_{1012}$ & $=1$ 
\\ \hline
level $l=2$ &
$z^{1010}_{1010}X_{2011}$ & $w^{1010}_{1010}Y_{2011}$ & $=1$\\ \hline
level $l=3$ &
$z^{1010}_{2000}$ & $w^{1010}_{2000}$ & $=1$\\ \hline
\end{tabular}
\end{center}

\begin{center}
\setlength\extrarowheight{5pt}
\begin{tabular}{llllll}
\multicolumn{6}{c}{Cusp equations for $\lambda$} \\ \hline \hline
level $l=1$ &
$w^{0110}_{0200}Y_{0112}Y_{0211}$ & $z^{0101}_{0002}X_{0112}X_{0121}$ & $w^{1100}_{0200}Y_{1120}Y_{1210}$ &
$z^{1001}_{0002}X_{1102}X_{1201}$\\
&$w^{0110}_{0020}Y_{1120}Y_{2110}$& $z_{2000}^{1010}X_{1210}X_{2110}$ & $w^{0011}_{0020}Y_{0121}Y_{0211}$ &
$z^{1001}_{2000}X_{1021}X_{2011}$ & $=1$\\ \hline
level $l=2$ &
$w^{0110}_{0110}Y_{0121}$ & $z^{0101}_{0101}X_{0211}$ & $w^{1100}_{1100}Y_{2110}$ &
$z^{1001}_{1001}X_{2101}$ \\
& $w^{0110}_{0110}Y_{1210}$ & $z_{1010}^{1010}X_{1120}$ & $w^{0011}_{0011}Y_{0112}$ &
$z^{1001}_{1001}X_{1012}$ & $=1$\\ \hline
level $l=3$ &
$w^{0110}_{0020}$ & $z^{0101}_{0200}$ & $w^{1100}_{2000}$ &
$z^{1001}_{2000}$ \\
& $w^{0110}_{0200}$ & $z_{0020}^{1010}$ & $w^{0011}_{0002}$ &
$z^{1001}_{0002}$ & $=1$\\ \hline
\end{tabular}
\end{center}
\end{table}



\bibliographystyle{plain}
\bibliography{BibFile}

\end{document}